%% file: GJSCStar.tex
\documentclass[12pt]{article}
\pdfoutput=1

\usepackage{amsmath, amsthm, amssymb}
\usepackage{fullpage}
\usepackage{enumerate}
\usepackage{ifpdf}
\usepackage{mathrsfs}
\ifpdf
\usepackage[pdftex]{graphicx}
\else
\usepackage[dvips]{graphicx}
\fi
\usepackage{tikz}
 	 \usetikzlibrary{arrows,backgrounds}
\usepackage[all]{xy}

\input xy
\xyoption{all}

\usepackage[pdftex,plainpages=false,hypertexnames=false,pdfpagelabels]{hyperref}
\newcommand{\arXiv}[1]{\href{http://arxiv.org/abs/#1}{\tt arXiv:\nolinkurl{#1}}}
\newcommand{\arxiv}[1]{\href{http://arxiv.org/abs/#1}{\tt arXiv:\nolinkurl{#1}}}

\newcommand{\doi}[1]{\href{http://dx.doi.org/#1}{{\tt DOI:#1}}}

\newcommand{\googlebooks}[1]{(preview at \href{http://books.google.com/books?id=#1}{google books})}

\usepackage{xcolor}
\definecolor{dark-red}{rgb}{0.7,0.25,0.25}
\definecolor{dark-blue}{rgb}{0.15,0.15,0.55}
\definecolor{medium-blue}{rgb}{0,0,.8}
\definecolor{DarkGreen}{RGB}{0,150,0}
\hypersetup{
   colorlinks, linkcolor={purple},
   citecolor={medium-blue}, urlcolor={medium-blue}
}

\theoremstyle{plain}
\newtheorem{thm}{Theorem}[section]
\newtheorem*{thm*}{Theorem}
\newtheorem{alphathm}{Theorem}

\newtheorem{cor}[thm]{Corollary}
\newtheorem*{cor*}{Corollary}
\newtheorem{lem}[thm]{Lemma}

\newtheorem{prop}[thm]{Proposition}

\newtheorem*{quest*}{Question}

\theoremstyle{definition}
\newtheorem{defn}[thm]{Definition}
\newtheorem{assumption}[thm]{Assumption}
\newtheorem{nota}[thm]{Notation}

\newtheorem{ex}[thm]{Example}
\newtheorem{rem}[thm]{Remark}

\newtheorem{fact}[thm]{Fact}
\newtheorem{facts}[thm]{Facts}

\DeclareMathOperator{\dist}{dist}

\DeclareMathOperator{\Epi}{Epi}

\DeclareMathOperator{\depth}{depth}

\DeclareMathOperator{\Gr}{Gr}

\DeclareMathOperator{\id}{id}

\DeclareMathOperator{\op}{op}

\DeclareMathOperator{\spann}{span}

\DeclareMathOperator{\Tr}{Tr}
\DeclareMathOperator{\tr}{tr}

\newcommand{\comment}[1]{}

\newcommand{\be}{\begin{enumerate}[(1)]}
\newcommand{\ee}{\end{enumerate}}

\newcommand{\N}{\mathbb{N}}
\newcommand{\Z}{\mathbb{Z}}

\newcommand{\F}{\mathbb{F}}
\newcommand{\R}{\mathbb{R}}

\newcommand{\C}{\mathbb{C}}

\newcommand{\I}{\infty}
\newcommand{\set}[2]{\left\{#1 \middle| #2\right\}}
\newcommand{\ketbra}[2]{|#1\rangle \langle #2|}

\newcommand{\e}{\epsilon}

\newcommand{\noshow}[1]{}
\newcommand{\MR}[1]{}

\newcommand{\jw}[1]{f^{(#1)}}

\newcommand{\Asterisk}{\mathop{\scalebox{1.5}{\raisebox{-0.2ex}{$*$}}}}%

\newcommand{\TL}{\cT\hspace{-.08cm}\cL}

\newcommand{\TP}[1]{\cT_{#1}(\cP_\bullet)}
\newcommand{\OP}[1]{\cO_{#1}(\cP_\bullet)}
\newcommand{\SP}[1]{\cS_{#1}(\cP_\bullet)}
\newcommand{\FP}[1]{\cF_{#1}(\cP_\bullet)}

\def\semicolon{;}
\def\applytolist#1{
    \expandafter\def\csname multi#1\endcsname##1{
        \def\multiack{##1}\ifx\multiack\semicolon
            \def\next{\relax}
        \else
            \csname #1\endcsname{##1}
            \def\next{\csname multi#1\endcsname}
        \fi
        \next}
    \csname multi#1\endcsname}

\def\calc#1{\expandafter\def\csname c#1\endcsname{{\mathcal #1}}}
\applytolist{calc}QWERTYUIOPLKJHGFDSAZXCVBNM;
\def\bbc#1{\expandafter\def\csname bb#1\endcsname{{\mathbb #1}}}
\applytolist{bbc}QWERTYUIOPLKJHGFDSAZXCVBNM;
\def\bfc#1{\expandafter\def\csname bf#1\endcsname{{\mathbf #1}}}
\applytolist{bfc}QWERTYUIOPLKJHGFDSAZXCVBNM;
\def\sfc#1{\expandafter\def\csname s#1\endcsname{{\sf #1}}}
\applytolist{sfc}QWERTYUIOPLKJHGFDSAZXCVBNM;
\def\ffc#1{\expandafter\def\csname f#1\endcsname{{\mathfrak #1}}}
\applytolist{ffc}QWERTYUIOPLKJHGFDSAZXCVBNM;

\usepackage{yfonts}

\tikzstyle{shaded}=[fill=red!10!blue!20!gray!30!white]
\tikzstyle{unshaded}=[fill=white]
\tikzstyle{empty box}=[circle, draw, thick, fill=white, opaque, inner sep=2mm]
\tikzstyle{annular}=[scale=.7, inner sep=1mm, baseline]
\tikzstyle{rectangular}=[scale=.75, inner sep=1mm, baseline=-.1cm]
\usetikzlibrary{calc}

\newcommand{\nbox}[6]{
	\draw[thick, #1] ($#2+(-#3,-#3)+(-#4,0)$) rectangle ($#2+(#3,#3)+(#5,0)$);
	\coordinate (ZZa) at ($#2+(-#4,0)$);
	\coordinate (ZZb) at ($#2+(#5,0)$);
	\node at ($1/2*(ZZa)+1/2*(ZZb)$) {#6};
}

\begin{document}

\title{$C^*$-algebras from planar algebras II:
\\
{\large the Guionnet-Jones-Shlyakhtenko $C^*$-algebras}
}
\author{Michael Hartglass and David Penneys}
\date{\today}
\maketitle
\begin{abstract}
We study the $C^*$-algebras arising in the construction of Guionnet-Jones-Shlyakhtenko (GJS) for a planar algebra.
In particular, we show they are pairwise strongly Morita equivalent, we compute their $K$-groups, and we prove many properties, such as simplicity, unique trace, and stable rank 1.
Interestingly, we see a $K$-theoretic obstruction to the GJS $C^*$-algebra analog of Goldman-type theorems for II$_1$-subfactors.
This is the second article in a series studying canonical $C^*$-algebras associated to a planar algebra.
\end{abstract}


\input{Chapters/IntroductionPartII.tex}

\input{Chapters/HilbertBimodule.tex}

\input{Chapters/GradedAndFiltered.tex}

\input{Chapters/RieffelMorita.tex}

\input{Chapters/GJSProperties.tex}

\input{Chapters/CStarTower.tex}

\bibliographystyle{amsalpha}

{\footnotesize
\bibliography{../bibliography}
}
\end{document}

%% file: Chapters/IntroductionPartII.tex

\section{Introduction}

A finite index subfactor $N\subset M$ is studied by analyzing its standard invariant, which has been axiomatized in three similar ways, each emphasizing slightly different structure: Ocneanu's paragroups \cite{MR996454,MR1642584}, Popa's $\lambda$-lattices \cite{MR1334479}, and Jones' planar algebras \cite{math/9909027}.

In \cite{MR1198815,MR1334479,MR1887878}, Popa starts with a $\lambda$-lattice $A_{\bullet,\bullet}=(A_{i,j})$ and constructs a {\rm II}$_1$-subfactor whose standard invariant is $A_{\bullet,\bullet}$.
In \cite{MR2051399}, Popa and Shlyakhtenko identified the factors in certain cases of Popa's reconstruction theorems.
They showed every subfactor planar algebra arises as the standard invariant of some subfactor $N\subset M$ such that $N,M$ are both isomorphic to $L(\F_\infty)$.

These results were recently reproduced using a combination of planar algebra and free probability techniques.
In \cite{MR2732052,MR2807103}, starting with a subfactor planar algebra $\cP_\bullet$, Guionnet-Jones-Shlyakhtenko (GJS) gave a diagrammatic proof of Popa's celebrated reconstruction theorem, and they showed that the resulting factors are interpolated free group factors when $\cP_\bullet$ is finite depth.
When $\cP_\bullet$ is infinite depth, Hartglass showed the factors are $L(\F_\infty)$ \cite{MR3110503}.

From the subfactor planar algebra $\cP_\bullet$, GJS use a graded multiplication $\wedge_k$ and the Voiculescu trace $\tau_k$ to form a tower of graded algebras $(\Gr_k)_{k\geq 0}$ which completes to a Jones tower of {\rm II}$_1$-factors $(\cM_k)_{k\geq 0}$ which recovers the planar algebra.

$$
\tau_k(x\wedge_k y)
=
\tau_k\left(
\begin{tikzpicture}[baseline = .1cm]
	\draw (-.8, 0)--(2, 0);
	\draw (0, 0)--(0, .8);
	\draw (1.2, 0)--(1.2, .8);
	\nbox{unshaded}{(0,0)}{.4}{0}{0}{$x$}
	\nbox{unshaded}{(1.2,0)}{.4}{0}{0}{$y$}
	\node at (-.6, .2) {{\scriptsize{$k$}}};
	\node at (.6, .2) {{\scriptsize{$k$}}};
	\node at (1.8, .2) {{\scriptsize{$k$}}};
\end{tikzpicture}
\right)
=
\delta^{-k}
\begin{tikzpicture}[baseline=.3cm]
	\draw (.4,0)--(.4,.6);
	\draw (-.4,0)--(-.4,.6);
	\draw (-.7,0) arc (90:270:.3cm) -- (.7,-.6) arc (-90:90:.3cm) -- (-.7,0);
	\nbox{unshaded}{(-.4,0)}{.3}{0}{0}{$x$}
	\nbox{unshaded}{(.4,0)}{.3}{0}{0}{$y$}
	\nbox{unshaded}{(0,.8)}{.3}{.4}{.4}{$\Sigma \, \TL$}
	\node at (1.1, -.4) {{\scriptsize{$k$}}};		
\end{tikzpicture}
$$

Starting with the tower of graded algebras $(\Gr_k)_{k\geq 0}$, we could instead look at the $C^*$-algebra completion to get a tower of $C^*$-algebras $(\cA_k)_{k\geq 0}$.
Two motivations for doing so are as follows:
\be
\item
\underline{Non-commutative topology:}
A unital abelian $C^*$-algebra $A$ is isomorphic to $C(X)$ for a compact Hausdorff $X$, and the space $X$ can be recovered from $A$ via its characters.
Thus an arbitrary $C^*$-algebra can be thought of as functions on a non-commutative topological space.
However, taking the von Neumann completion of a $C^*$-algebra forgets the non-commutative topology remembered by the $C^*$-algebra.
Hence the algebras $\cA_k$ in the $C^*$-tower remember topological invariants of the planar algebra $\cP_\bullet$, e.g. $K$-theory, while still allowing us to recover $\cP_\bullet$ from the relative commutants.
\item
\underline{Non-commutative geometry:}
For the planar algebra of non-commutative polynomials, where tangles act by contracting indices, $\Gr_0$ is an algebra of non-commutative polynomials, and $\cA_0$ is Voiculescu's reduced $C^*$-algebra generated by free semi-circular elements.
Thus we may view $\Gr_0$ as an algebra of non-commutative polynomials inside the algebra of non-commutative continuous functions $\cA_0$.
This is precisely the situation to look for non-commutative geometry \`{a} la Connes via Dirac operators and spectral triples \cite{MR1303779}.
\ee

This article addresses the first motivation listed above: determining properties of the $\cA_k$'s, and computing $K$-theory.
We work primarily with an unshaded factor planar algebra $\cP_\bullet$ with one strand type, and we provide remarks to translate to the case of a shaded subfactor planar algebra.
The principal graph of $\cP_\bullet$ is denoted $\Gamma$.

Our main tool is the functorial construction from the first part of this series \cite{CStarFromPAs}, based on Pimsner's algebras in \cite{MR1426840}.
In \cite{CStarFromPAs}, we constructed a canonical Hilbert $C^*$-bimodule $\cX(\cP_\bullet)$ associated to a planar algebra $\cP_\bullet$, producing the Pimsner-Toeplitz algebra $\TP{}$, the Cuntz-Pimsner algebra $\OP{}$ \cite{MR1426840}, and the free semicircular algebra $\SP{}$. 
We will see in Lemma \ref{lem:AInftyIso} that $\SP{}$ is isomorphic to the semi-fnite GJS algebra \cite{MR2807103} defined below in Subsection \ref{sec:GJSConstruction}.

Taking (reduced) compressions of these algebras yields numerous canonical $C^*$-algebras, including the universal (Toeplitz-)Cuntz-Krieger algebras $\cT_{\vec{\Gamma}},\cO_{\vec{\Gamma}}$ \cite{MR561974,MR1722197}, the free semicircular graph algebra $\cS(\Gamma)$ \cite{CStarFromPAs}, the Doplicher-Roberts algebra $\cO_\rho$ \cite{MR1010160}, and the zeroth GJS algebra $\SP{0}\cong\cA_0$ \cite{MR2732052,MR2645882}.
In fact, all the $\cA_k$ arise as compressions of the semifinite algebra.
See Section \ref{sec:GradedAndFiltered} and \cite{CStarFromPAs} for more details.

Using the free graph algebra $\cS(\Gamma)$, we prove the following main theorems about the $\cA_k$.
First, for the Jones tower $(\cM_k)_{k\geq 0}$, we know that $\cM_k$ is strongly Morita equivalent \cite{MR0353003,MR0367670} to $\cM_{k+2}$ for all $k\geq 0$ for a shaded subfactor planar algebra.
For unshaded factor planar algebras, we get the following theorem.

\begin{alphathm}\label{thm:MainRME}
The algebras $\cA_k$ are pairwise strongly Morita equivalent.
\end{alphathm}

When $\cP_\bullet$ is shaded, we get two flavors of algebras: those equivalent to $\cA_{0,+}$, and those equivalent to $\cA_{0,-}$, where the $\pm$ refers to the shading.
More precisely $\cA_{n,\pm}$ is strongly Morita equivalent to $\cA_{n+1,\mp}$ for all $n\geq 0$.
When $\cP_\bullet$ is finite depth, Theorem \ref{thm:MainRME} has a simple proof for $k>\depth(\cP_\bullet)-1$ using a Pimsner-Popa basis for $\cP_{n}$ over $\cP_{n-1}$.

Second, since the $\cA_k$'s are all compressions of $\SP{}$, we have the following.

\begin{alphathm}
$K_0(\cA_k)\cong \bbZ[\set{[p_\alpha]}{\alpha\in V(\Gamma)}]$ and $K_1(\cA_k)=(0)$.
\end{alphathm}

As a corollary of Theorem \ref{thm:MainRME}, we get the existence of Pimsner-Popa bases for $\cA_{k+1}$ over $\cA_k$, but interestingly, there are $K$-theoretic obstructions to the size of these bases.
We immediately see that $\cA_k\subset \cA_{k+1}$ has the same Watatani index \cite{MR996807} as $\cM_k\subset \cM_{k+1}$, which must be equal to the Jones index \cite{MR696688}.
See Section \ref{sec:CStarTower} for more details.
Interestingly, we get the following corollary.

\begin{cor*}
For the planar algebra $\cP_\bullet$ associated to the group subfactor $M\subset M\rtimes G$ for a finite group $G$, $\cA_1$ is not isomorphic to $\cA_0\rtimes G$ (see Example \ref{ex:GroupAction}).
\end{cor*}

Next, using free probability techniques, we prove the following.

\begin{alphathm}\label{thm:Properties}
$\cA_0$ is simple, has unique tracial state $\Tr$, has stable rank 1, and
\begin{align*}
K_0(\cA_0)^+
&= 
\set{x\in K_0(\cA_k)}{\Tr(x)>0}\cup\{0\} 
\text{ and }\\
\Sigma(\cA_0) 
&= 
\set{[p]}{p \in P(\cA_0)}= \set{x \in K_{0}(\cA_{0})}{ \tr(x) \in (0, 1) } \cup \{[1], [0]\}.
\end{align*}
\end{alphathm}

Surprisingly, the above theorem gives a strict dichotomy depending on the quantum dimensions of the vertices of $\Gamma$.

\begin{cor*}
If $\Tr(p_\alpha)\in \bbN$ for all $\alpha\in V(\Gamma)$, then $\cA_0$ is projectionless.
Otherwise $\Tr(\Sigma(\cA_0))$ is dense in $[0,1]$.
\end{cor*}

\subsection{Outline}

In Section \ref{sec:Canonical}, we briefly recall the canonical constructions from Part I \cite{CStarFromPAs}.
In Section \ref{sec:GradedAndFiltered}, we discuss the graded (GJS, Voiculescu) and filtered (Walker) conventions for the algebras arising from \cite{MR2732052,MR2645882}.
We prove our results on strong Morita equivalence in Section \ref{sec:Morita}, and we use free probability techniques to prove Theorem \ref{thm:Properties} in Section \ref{sec:properties}.
Finally, we discuss the $C^*$-Watatani tower $(\cA_k)_{k\geq 0}$ in Section \ref{sec:CStarTower}, along with $K$-theoretic obstructions.

\subsection{Future research}

First, we would like to investigate further properties of $\cA_0$. 
For example, it would be interesting to determine its Cuntz semigroup and determine when it has real rank zero.  
The fact that the trace of the scale of $\cA_{0}$ is often dense in $[0, 1]$ gives positive evidence to $\cA_{0}$ having real rank zero in many cases.  

Second, we would like to address the second motivation posed in the introduction, and study spectral triples associated to $\cA_{0}$. 
As an example, the number operator which maps $x\in \cP_n\subset \Gr_0$ to $nx$ is an example of a $\theta$-summable Dirac operator $D$ with compact resolvent. 
Incorporating a factor of $\delta^{2n}$ corresponding to the grading on $\Gr_0$ makes it $p$-summable for all $p>1$.
In particular, we would like to compute the Chern character of $D$, the associated cyclic cocycle, and the induced topology on the state space of $\cA_0$.
It would also be interesting to compute $K^*(\cA_0)$ and its pairing with $K_*(\cA_0)$.

\subsection{Acknowledgements}

The authors would like to thank Ken Dykema, George Elliott, Vaughan Jones, Dima Shlyakhtenko, and Dan Voiculescu for many helpful conversations.
David Penneys was partially supported by the Natural Sciences and Engineering Research Council of Canada.
Both authors were supported by DOD-DARPA grant HR0011-12-1-0009.

%% file: Chapters/HilbertBimodule.tex

\section{The canonical $C^*$-Hilbert bimodule of $\cP_\bullet$}\label{sec:Canonical}

We now rapidly recall the construction from Part I \cite{CStarFromPAs}.
We refer the reader to \cite{1208.5505} for the definition of a factor planar algebra and its principal graph.
A subfactor planar algebra is a shaded factor planar algebra (see \cite{MR2972458} for more details).
As in \cite{1208.5505,CStarFromPAs}, we use the convention that the outer rectangle is omitted, and the $\star$ on each rectangle is on the bottom.

In this article, $\cP_\bullet$ is a (sub)factor planar algebra with $\delta>1$ and $\Gamma$ is its principal graph.
We usually assume $\cP_\bullet$ is unshaded, and we provide remarks along the way for the alterations needed in the shaded case.

\subsection{The $C^*$-Hilbert bimodule $\FP{}$}\label{sec:X}

\begin{defn}[The ground algebra $\cB$]
For $n\geq 0$, let $\cB_{n}=\cB_n(\cP_\bullet)$ be the external direct sum
$$
\cB_n=\bigoplus_{l, r = 0}^{n} \cP_{l, r},
$$
where $\cP_{l, r} = \cP_{l + r}$, and we picture an element of $b\in \cP_{l, r}$ as
$
\begin{tikzpicture}[baseline = -.1 cm]
	\draw (-.6, 0)--(.6, 0);
	\nbox{unshaded}{(0,0)}{.3}{0}{0}{$b$}
	\node at (-.5, .2) {{\scriptsize{$l$}}};
	\node at (.5, .2) {{\scriptsize{$r$}}};
\end{tikzpicture}
$.

We define a multiplication and (non-normalized) trace $\Tr$ on $\cB_{n}$ by
$$
\begin{tikzpicture}[baseline = 0cm]
	\draw (-.8, 0)--(.8, 0);
	\nbox{unshaded}{(0,0)}{.4}{0}{0}{$a$}
	\node at (-.6, .2) {{\scriptsize{$l$}}};
	\node at (.6, .2) {{\scriptsize{$r$}}};
\end{tikzpicture}
\, \cdot \,
\begin{tikzpicture}[baseline = 0cm]
	\draw (-.8, 0)--(.8, 0);
	\nbox{unshaded}{(0,0)}{.4}{0}{0}{$b$}
	\node at (-.6, .2) {{\scriptsize{$l'$}}};
	\node at (.6, .2) {{\scriptsize{$r'$}}};
\end{tikzpicture}\,
= \, \delta_{r, l'} \,
\begin{tikzpicture}[baseline = 0cm]
	\draw (-.8, 0)--(2, 0);
	\nbox{unshaded}{(0,0)}{.4}{0}{0}{$a$}
	\node at (-.6, .2) {{\scriptsize{$l$}}};
	\node at (.6, .2) {{\scriptsize{$r$}}};
    \nbox{unshaded}{(1.2,0)}{.4}{0}{0}{$b$}
	\node at (1.8, .2) {{\scriptsize{$r'$}}};
\end{tikzpicture}
\, \, \text{ and } \, \,
\Tr(b) = \delta_{l, r}\,
\begin{tikzpicture}[baseline = -.2cm]
	\draw (-.4, 0)--(.4, 0) arc(90:-90:.4cm)--(-.4, -.8) arc(270:90:.4cm);
	\nbox{unshaded}{(0,0)}{.4}{0}{0}{$b$}
	\node at (-.6, .2) {{\scriptsize{$l$}}};
	\node at (.6, .2) {{\scriptsize{$r$}}};
\end{tikzpicture}\,.
$$
We equip $\cB_{n}$ with involution $\dagger$ defined as follows.
If $b \in \cP_{l, r}$, then
$$
b^{\dagger} =
\begin{tikzpicture}[baseline = 0cm]
	\draw (-.8, 0)--(.8, 0);
	\nbox{unshaded}{(0,0)}{.4}{0}{0}{$b^{*}$}
	\node at (-.6, .2) {{\scriptsize{$r$}}};
	\node at (.6, .2) {{\scriptsize{$l$}}};
\end{tikzpicture}\,
\in  \cP_{r, l},
$$
where the $*$ is the standard involution on $\cP_\bullet$.
Notice that $\cB_{n}$ is a finite dimensional $C^{*}$-algebra under $\dagger$, so $\cB_n$ has a unique $C^{*}$-norm.  
Since $\cB_{n}$ is canonically a hereditary $C^*$-subalgebra of $\cB_{n+1}$, we define $\cB=\varinjlim \cB_n$.
\end{defn}

\begin{defn}[The bimodules $\cX_{n}$]
Fixing $n \geq 0$, we define 
$$
X_{n} = \bigoplus_{l, r = 1}^{\I} \cP_{l, n, r}
$$
where diagrammatically, an element $x \in \cP_{l, n, r}$ can be seen as 
$
x = \begin{tikzpicture}[baseline = -.1 cm]
	\draw (-.6, 0)--(.6, 0);
	\draw (0, 0)--(0, .6);
	\nbox{unshaded}{(0,0)}{.3}{0}{0}{$x$}
	\node at (-.5, .2) {{\scriptsize{$l$}}};
	\node at (.5, .2) {{\scriptsize{$r$}}};
	\node at (.15, .45) {{\scriptsize{$n$}}};
\end{tikzpicture}
$.
The vector space $X_{n}$ comes equipped with a $\cup_{k\geq 0}\cB_k$-valued inner product which is the sesquilinear extension of
$$
\langle x| y \rangle_{\cB}
= \delta_{l, l'}\,
\begin{tikzpicture}[baseline = 0cm]
	\draw (-.8, 0)--(2, 0);
	\draw (0, .4) arc(180:90:.3cm) -- (.9,.7) arc (90:0:.3cm);
	\node at (.6, .85) {\scriptsize{$n$}};
	\nbox{unshaded}{(0,0)}{.4}{0}{0}{$x^{*}$}
	\node at (-.6, .2) {{\scriptsize{$r$}}};
	\node at (.6, .2) {{\scriptsize{$l$}}};
	\nbox{unshaded}{(1.2,0)}{.4}{0}{0}{$y$}
	\node at (1.8, .2) {{\scriptsize{$r'$}}};
\end{tikzpicture}
$$
for $x \in \cP_{l,n, r}$ and $y \in P_{l', n, r'}$.  
The algebra $\cup_{k\geq 0} \cB_{k}$ naturally acts on the left and right, namely
\begin{align*}
\begin{tikzpicture}[baseline = -.1 cm]
	\draw (-.8, 0)--(.8, 0);
	\nbox{unshaded}{(0,0)}{.4}{0}{0}{$b$}
	\node at (-.6, .2) {{\scriptsize{$l$}}};
	\node at (.6, .2) {{\scriptsize{$r$}}};
\end{tikzpicture}
\,\cdot \,
\begin{tikzpicture}[baseline = -.1 cm]
	\draw (-.8, 0)--(.8, 0);
	\draw (0, 0)--(0, .8);
	\nbox{unshaded}{(0,0)}{.4}{0}{0}{$x$}
	\node at (-.6, .2) {{\scriptsize{$l'$}}};
	\node at (.6, .2) {{\scriptsize{$r'$}}};
\end{tikzpicture}\,
&= \delta_{r, l'} \cdot
\begin{tikzpicture}[baseline = -.1 cm]
	\draw (-.8, 0)--(2, 0);
	\draw (1.2, 0)--(1.2, .8);
	\nbox{unshaded}{(0,0)}{.4}{0}{0}{$b$}
	\node at (-.6, .2) {{\scriptsize{$l$}}};
	\node at (.6, .2) {{\scriptsize{$r$}}};
	\node at (1.8, .2) {{\scriptsize{$r'$}}};
	\nbox{unshaded}{(1.2,0)}{.4}{0}{0}{$x$};
\end{tikzpicture}\, \text{ and }
\\
\begin{tikzpicture}[baseline = -.1 cm]
	\draw (-.8, 0)--(.8, 0);
    \draw (0, 0)--(0, .8);
	\nbox{unshaded}{(0,0)}{.4}{0}{0}{$x$}
	\node at (-.6, .2) {{\scriptsize{$l'$}}};
	\node at (.6, .2) {{\scriptsize{$r'$}}};
\end{tikzpicture}\, \cdot \, \begin{tikzpicture}[baseline = -.1 cm]
	\draw (-.8, 0)--(.8, 0);
	\nbox{unshaded}{(0,0)}{.4}{0}{0}{$b$}
	\node at (-.6, .2) {{\scriptsize{$l$}}};
	\node at (.6, .2) {{\scriptsize{$r$}}};
\end{tikzpicture}\,
&=
\delta_{r', l} \cdot
\begin{tikzpicture}[baseline = -.1 cm]
	\draw (-.8, 0)--(2, 0);
	\draw (0, 0)--(0, .8);
	\nbox{unshaded}{(1.2,0)}{.4}{0}{0}{$b$}
	\node at (-.6, .2) {{\scriptsize{$l'$}}};
	\node at (.6, .2) {{\scriptsize{$l$}}};
	\node at (1.8, .2) {{\scriptsize{$r$}}};
	\nbox{unshaded}{(0,0)}{.4}{0}{0}{$x$};
\end{tikzpicture}\, .
\end{align*}
We complete $X_{n}$ to form $\cX_{n} = \cX_{n}(\cP_{\bullet})$, which is naturally a $\cB-\cB$ Hilbert bimodule.  
There is an involution $\dagger$ on $\cX_{n}$ given by the continuous conjugate-linear extension of 
$$
x^\dagger = \begin{tikzpicture}[baseline = 0cm]
	\draw (-.8, 0)--(.8, 0);
	\draw (0, 0)--(0, .8);
	\nbox{unshaded}{(0,0)}{.4}{0}{0}{$x^{*}$}
        \node at (.2, .6) {{\scriptsize{$n$}}};
	\node at (-.6, .2) {{\scriptsize{$r$}}};
	\node at (.6, .2) {{\scriptsize{$l$}}};
\end{tikzpicture} \in \cP_{r, n, l}\,.
$$
for $x \in \cP_{l, n, r}$.
We write $\cX = \cX_{1}$, and note $\cB=\cX_0$.
\end{defn}

\begin{defn}[The Fock space $\FP{}$]
The \underline{Pimsner-Fock space of $\cP_{\bullet}$} $\FP{}$ is the $\cB-\cB$ Hilbert bimodule
$$
\FP{} = \bigoplus_{n=0}^{\infty} \cX_{n}.
$$
We define creation and annihilation operators on $\FP{}$ as follows.
Let $x \in \cP_{l,1,r}\subset \cX$.
We define the $x$ creation operator $L_+(x)$ on $\FP{}$ by the linear extension of its action on $y\in \cP_{l',n,r'}$:
$$
L_+(x)\cdot
\left(
\begin{tikzpicture}[baseline=.1cm]
	\draw (0,0)--(0,.8);
	\draw (-.8, 0)--(.8, 0);
	\node at (-.6,.2) {\scriptsize{$l'$}};
	\node at (.6,.2) {\scriptsize{$r'$}};
	\node at (.2,.6) {\scriptsize{$n$}};	
	\nbox{unshaded}{(0,0)}{.4}{0}{0}{$y$}
\end{tikzpicture}
\right)
=
\delta_{r,l'}\,
\begin{tikzpicture}[baseline = .1cm]
	\draw (-.8, 0)--(2, 0);
	\draw (0, 0)--(0, .8);
	\draw (1.2, 0)--(1.2, .8);
	\node at (-.6,.2) {\scriptsize{$l$}};
	\node at (.6,.2) {\scriptsize{$r$}};
	\node at (1.8,.2) {\scriptsize{$r'$}};
	\node at (1.4,.6) {\scriptsize{$n$}};
	\nbox{unshaded}{(0,0)}{.4}{0}{0}{$x$}
	\nbox{unshaded}{(1.2,0)}{.4}{0}{0}{$y$}
\end{tikzpicture}\,.
$$
We define the $x$ annihilation operator $L_-(x)$ on $\FP{}$ by the extension of
$$
L_-(x)\cdot
\left(
\begin{tikzpicture}[baseline=-.1cm]
	\draw (0,0)--(0,.8);
	\draw (-.8, 0)--(.8, 0);
	\node at (-.6,.2) {\scriptsize{$l'$}};
	\node at (.6,.2) {\scriptsize{$r'$}};
	\node at (.2,.6) {\scriptsize{$n$}};	
	\nbox{unshaded}{(0,0)}{.4}{0}{0}{$y$}
\end{tikzpicture}
\right)
=
\delta_{r,l'}\,
\begin{tikzpicture}[baseline = -.1cm]
	\draw (-.8, 0)--(2, 0);
	\draw (0.2, .4) arc (180:0:.4cm);
	\draw (1.2, 0)--(1.2, .8);
	\node at (-.6,.2) {\scriptsize{$l$}};
	\node at (.6,.2) {\scriptsize{$r$}};
	\node at (1.8,.2) {\scriptsize{$r'$}};
	\node at (1.6,.6) {\scriptsize{$n-1$}};
	\nbox{unshaded}{(0,0)}{.4}{0}{0}{$x$}
	\nbox{unshaded}{(1.2,0)}{.4}{0}{0}{$y$}
\end{tikzpicture}\,.$$
Note that $(L_{+}(x))^{*} = L_{-}(x^{\dagger})$.
\end{defn}

In \cite{CStarFromPAs}, we showed that the $C^{*}$-algebra generated by $\cB$ and the operators $L_{+}(x)$ is isomorphic to the Pimsner-Toeplitz algebra $\cT(\cX)$ \cite{MR1426840}.  

\begin{defn}
The \underline{free semicircular algebra of $\cP_{\bullet}$} $\SP{}$ is defined to be the $C^*$-algebra generated by $\cB$ together with the elements $L_{+}(x) + L_{-}(x)$ where $x \in \cX$ and $x = x^{\dagger}$.  
\end{defn}

\begin{facts}\label{rem:WalkerMult}
\mbox{}
\be
\item
The work of Germain \cite{Germain} allowed us to use this presentation to compute the $K$-groups of $\SP{}$.
Recalling that $\Gamma$ is the principal graph of $\cP_{\bullet}$, we have
$$
K_0(\SP{})\cong \bbZ[\set{[p_\alpha]}{\alpha\in V(\Gamma)}] \text{ and } K_1(\SP{})=(0).
$$

\item 
In \cite{CStarFromPAs}, we showed that $\SP{}$ can be pictured as generated by elements $x \in \cP_{l, n, r}$ where the action on $\FP{}$ is given by 
$$
x \cdot 
\left(
\begin{tikzpicture}[baseline = .1cm]
	\draw (-.8, 0)--(.8, 0);
	\draw (0, 0)--(0, .8);
    \node at (-.2, .6) {\scriptsize{$m$}};
	\nbox{unshaded}{(0,0)}{.4}{0}{0}{$y$}
\end{tikzpicture}
\right) = \sum_{k=0}^{\min\{n, m\}}\, 
\begin{tikzpicture} [baseline = -.1cm]
	\draw (-.8, 0)--(2, 0);
	\draw (-.2, 0)--(-.2, .8);
	\draw (.2,.4) arc(180:0:.4cm);
	\draw (1.4, 0)--(1.4, .8);
	\node at (.6, .6) {\scriptsize{$k$}};
	\nbox{unshaded}{(0,0)}{.4}{0}{0}{$x$}
	\nbox{unshaded}{(1.2,0)}{.4}{0}{0}{$y$}
\end{tikzpicture}\,,
$$
and this diagram also represents how to multiply $x$ and $y$ in $\SP{}$.

\item 
Let $1_{k}$ be the diagram in $\cB$ which consists of $k$ horizontal through strings.  In \cite{CStarFromPAs}, we showed that $1_{0}\SP{} 1_{0} \cong \cA_{0}$ (see Section \ref{sec:GradedAndFiltered} below).  Similar arguments show that $1_{k}\SP{} 1_{k} \cong \cA_{k}$ for all $k \geq 0$.
  
\ee
\end{facts}

\subsection{Free graph algebras}\label{sec:FreeGraph}

We begin with some notation which will be useful for the rest of this artilce.

\begin{nota}
Let $\cA$ be a $C^{*}$ algebra equipped with a lower semicontinuous tracial weight $\Tr$ (which might not be finite).  
We write
$$
\cA = \overset{p_{1}}{\underset{\mu_{1}}{\cA_{1}}} \oplus \cdots \oplus \overset{p_{n}}{\underset{\mu_{n}}{\cA_{n}}} \oplus \cD
$$
if $p_{k}$ is a projection in $\cA$ which serves as the identity of the $C^{*}$-algebra $\cA_{k}$, and $\Tr(p_{k}) = \mu_{k}$.  
The algebra $\cD$ may be unital or not.  
\end{nota}

Let $\Lambda$ be an undirected graph with a weighting $\mu: V(\Lambda) \rightarrow\bbR_{>0}$.  
We place a semi-finite trace $\Tr$ on $C_{0}(V(\Lambda))$ by requiring $\Tr(p_{\alpha}) = \mu(\alpha)$ with $p_{\alpha}$ the indicator function at $\alpha$.  
We now construct the free graph algebra of $\Lambda$ \cite{CStarFromPAs}, which for our purposes, will be defined as a reduced amalgamated free product.

\begin{defn}
For $\e \in E(\Lambda)$, we will define algebras $\cS_{\e}$ with trace $\Tr_{\e}$ as follows:
\be

\item If $\e$ is a loop at $\alpha$, then  
$$
(\cS_{\e}, \Tr_{e}) = \overset{p_{\alpha}}{\underset{\mu(\alpha)}{(C[0, 1], d\lambda)}} \oplus C_{0}(V(\Lambda) \setminus \{\alpha\}, \mu)
$$
where $d\lambda$ is integration against Lebesgue measure.

\item If the endpoints of $\e$ are $\alpha$ and $\beta$ with $\alpha \neq \beta$ and $\mu(\alpha) = \mu(\beta)$, then 
$$
(\cS_{\e}, \Tr_{\e}) = \overset{q_{\alpha} + p_{\beta}}{\underset{2\mu(\beta)}{(M_{2}(\C) \otimes C([0, 1]), \Tr_{M_{2}(\C)} \otimes d\lambda}}) \oplus \overset{r_{\alpha}}{\underset{\mu(\alpha) - \mu(\beta)}{\C}} \oplus  C_{0}(V(\Lambda) \setminus \{\alpha, \beta\}, \mu)
$$
where $p_{\alpha} = q_{\alpha} + r_{\alpha}$ and $q_{\beta}$ and $p_{\alpha}$ are embedded along the diagonal of $M_{2}(\C)$

\item If $\e$ has $\alpha$ and $\beta$ as endpoints with $\alpha \neq \beta$ and $\mu(\alpha) = \mu(\beta)$, then set $\cD$ to be the following $C^{*}$-algebra
$$
\cD = \set{f: [0, 1] \rightarrow M_{2}(\C)}{f \text{ is continuous and } f(0) \text{ is diagonal}}
$$
then 
$$
(\cS_{e}, \Tr_{e}) = \overset{p_{\alpha} + p_{\beta}}{\underset{2\mu(\alpha)}{(\cD, \Tr_{M_{2}(\C)} \otimes d\lambda)}} \oplus C_{0}(V(\Lambda) \setminus \{\alpha, \beta\}, \mu).
$$
\ee
Each $\cS_{\e}$ comes with a $\Tr_{\e}$-preserving conditional expectation $E_{\e}$ onto $C_{0}(V(\Lambda))$, and $\Tr_{\e}$ restricted to $C_{0}(V(\Lambda))$ is integration against $\mu$.  We therefore define $\cS(\Lambda, \mu)$ as
$$
(\cS(\Lambda, \mu), E) = \underset{C_{0}(V(\Lambda))}{\Asterisk} (\cS_{\e}, E_{\e})_{\e \in E(\Lambda)}
$$
where the amalgamated free product is reduced.
\end{defn}

\begin{rem}
All the free products in this article are reduced, and we simply use the notation $\Asterisk$ to denote our reduced free products.
\end{rem}

From the free product construction, $\Tr$ extends to a lower semicontinuous semi-finite tracial weight on $\cS(\Lambda, \mu)$ by $\Tr \circ E$.  In \cite{CStarFromPAs}, we computed the $K$-groups of $\cS(\Lambda)$ using \cite{Germain}.
The following result will be used in the proofs in Subsection \ref{sec:Simple}.

\begin{fact}\label{fact:KTheory}
$K_0(\cS(\Lambda, \mu))\cong \bbZ[\set{[p_\alpha]}{\alpha\in V(\Lambda)}]$ and $K_1(\cS(\Lambda, \mu))=(0)$.
\end{fact}

\subsection{Compressing the system}\label{sec:compression}

For each $\alpha \in V(\Gamma)$, we choose a representative $p_{\alpha} \in \cP_{2\depth(\alpha)}\subset 1_{\depth(\alpha)} \cB 1_{\depth(\alpha)}$.  Let $\cC$ be the $C^{*}$ algebra generated by the $p_{\alpha}$ and note that $\cC\cong C_{0}(V(\Gamma))$.  The \underline{free graph algebra of $\cP_{\bullet}$} $\cS(\Gamma)$ is the $C^*$-algebra generated by $\cC \SP{} \cC$.

It was shown in \cite{CStarFromPAs}, based on ideas in \cite{MR2807103}, that $\cS(\Gamma) \cong \cS(\Gamma, \mu)$ as in Subsection \ref{sec:FreeGraph} for $\mu$ the quantum dimension weighting induced from the trace on $\cP_{\bullet}$, which satisfies the Frobenius-Perron condition for the modulus $\delta$.
It was also shown in \cite{CStarFromPAs} that $\SP{} = \cS(\Gamma) \otimes \cK$ for $\cK$ the algebra of compact operators on a separable, infinite-dimensional Hilbert space.

Given the principal graph $\Gamma$ of $\cP_{\bullet}$, we have an auxiliary directed graph $\vec{\Gamma}$ which is instrumental in our computations. 

\begin{defn}\label{defn:Directed}
Suppose $\Lambda$ is an undirected graph.
We define a directed graph $\vec{\Lambda}=(V(\vec{\Lambda}),E(\vec{\Lambda}),s,t)$ as follows:
\be
\item
$V(\Lambda) = V(\vec{\Lambda})$.
\item
For each $\epsilon\in E(\Lambda)$ which is a loop, i.e., $\epsilon$ only connects to one vertex $\alpha$, we have a single edge $\epsilon\in E(\vec{\Lambda})$ such that $s(\epsilon)=t(\epsilon)=\alpha$.
\item
For each $\epsilon\in E(\Lambda)$ which is not a loop, i.e., $\epsilon$ connects to two distinct vertices $\alpha,\beta$, we have two edges $\epsilon',\epsilon''\in E(\vec{\Lambda})$ such that $s(\epsilon')=t(\epsilon'')=\alpha$ and $t(\epsilon')=s(\epsilon'')=\beta$.
\ee
There is an involution $\op$ on $E(\vec{\Lambda})$ given by $\epsilon^{\op}=\epsilon$ if $\epsilon$ is a loop, and $(\epsilon')^{\op}=\epsilon''$ if $\epsilon$ is not a loop.
\end{defn}

\begin{rem}
One can also define the free graph algebra $\cS(\Lambda, \mu)$ from Subsection \ref{sec:FreeGraph} as a subalgebra of $\cT_{\vec{\Lambda}}$ \cite{MR1722197}, a Toeplitz extension of $\cO_{\vec{\Lambda}}$.
See \cite{CStarFromPAs} for more details.
\end{rem}

The point of the above definition is the highly useful \underline{edge elements}, defined as follows.

\begin{defn}[{\cite{CStarFromPAs}}]
Let $E(\alpha \rightarrow \beta)$ denote the edges in $\vec{\Gamma}$ with source $\alpha$ and target $\beta$.  
Let $\e \in E(\Gamma)$ and $\e'$ and $\e''$ as in Definition \ref{defn:Directed}.
\be
\item 
If $E(\alpha \rightarrow \alpha)$ is of size $n$, then $p_{\alpha}\cX p_{\alpha}$ is $n$-dimensional and has an orthonormal basis $\{g_{\e}\}$ indexed by $\e \in E(\alpha \rightarrow \alpha)$.  
We can and will choose the $g_{\e}$ to be self adjoint.

\item 
If $E(\alpha \rightarrow \beta)$ is of size $n$, then $p_{\alpha}\cX p_{\beta}$ is $n$-dimensional and has an orthonormal basis indexed by $g_{\e'}$ such that $\e \in E(\Gamma)$ with $\alpha$ and $\beta$ as endpoints and $s(\e') = \alpha$.  
We also have corresponding elements $g_{\e''} \in p_{\beta}\cX p_{\alpha}$.  
We choose the indexing so that $g_{\e'}^{\dagger} = g_{\e''}$.
\ee
\end{defn}

We have the following facts about the edge elements $g_{\e}$ and $g_{\e'}$, which help one explicitly realize the isomorphism $\cS(\Gamma)\cong \cS(\Gamma, \mu)$ from Subsection \ref{sec:FreeGraph}.

\begin{facts}[{\cite{MR2807103, CStarFromPAs}}]\label{facts:edge}
\mbox{}
\be

\item $\cS(\Gamma)$ is generated as a $C^{*}$-algebra by $\cC$ and the elements $g_{\e}$ and $g_{\e'}$ above where the multiplication is as in Facts \ref{rem:WalkerMult}.

\item There is a conditional expectation, $E: \cS(\Gamma) \rightarrow \cC$ which is defined on polynomials in $\cC$ and the $g_{e}$, $g_{\e'}$, which sends a polynomial to its term that lies in $\cC$.

\item If $\e$ is a loop, let $g_{\e}$ be as above, and if $\e$ is not a loop, let $h_{\e} = g_{\e'} + g_{\e''}$.  Then the elements $g_{\e},h_{\e}$ are all free with amalgamation with respect to $E$, and furthermore, $\Tr \circ E$ is a lower-semicontinuous tracial weight on $\cS(\Gamma)$ where $\Tr$ is as in Subsection \ref{sec:X}.

\item If $\e$ is a loop at $\alpha$ then the distribution of $g_{\e}^{2}$ in $p_{\alpha}\cS(\Gamma)p_{\alpha}$ with respect to $\Tr \circ E$ is Free-Poisson with no atoms whose support contains 0.

\item Let $\e \in E(\Gamma)$ with $\alpha$ and $\beta$ as endpoints, $\mu(\alpha) > \mu(\beta)$, and $s(\e') = t(\e'') = \alpha$.  The distribution of $g_{\e'}^{*}g_{\e'} = g_{\e''}g_{\e'}$ in $p_{\beta}\cS(\Gamma)p_{\beta}$ with respect to $\Tr \circ E$ is a Free-Poisson  with no atoms and is supported away from the origin.  In particular, $g_{\e'}^{*}g_{\e'}$ is invertible in $p_{\beta}\cS(\Gamma)p_{\beta}$. This means that the distribution of $g_{\e''}^{*}g_{\e''}$ in $p_{\alpha}\cS(\Gamma)p_{\alpha}$ has an atom of size $\mu(\alpha) - \mu(\beta)$ at 0 and the same absolutely continuous portion as the distribution of $g_{\e'}^{*}g_{\e'}$.

\item Let $\e \in E(\Gamma)$ with $\alpha$ and $\beta$ as endpoints, $\alpha \neq \beta$, $\mu(\alpha) = \mu(\beta)$, and $s(\e') = t(\e'') = \alpha$.  The distribution of $g_{\e'}^{*}g_{\e'}$ in $p_{\beta}\cS(\Gamma)p_{\beta}$ with respect to $\Tr \circ E$ is a Free-Poisson with no atoms whose support contains 0.  The same is true of the distribution of $g_{\e''}^{*}g_{\e''}$ in $p_{\alpha}\cS(\Gamma)p_{\alpha}$.
\ee
\end{facts}

%% file: Chapters/GradedAndFiltered.tex

\section{The graded and filtered conventions}\label{sec:GradedAndFiltered}

In this section, we recall the Guionnet-Jones-Shlyakhtenko construction \cite{MR2732052} and the orthogonal version \cite{MR2645882}.
The first constructs a graded algebra, and the second constructs a filtered algebra.
In Part I \cite{CStarFromPAs}, the free semicircluar algebras developed use the filtered convention described below.   After this section, we will frequently identify the $C^{*}$-algebras arising from these two conventions since they are isomorphic.  We will only explicitly state which convention we are using when it is needed. 

\subsection{The Guionnet-Jones-Shlyakhtenko construction}\label{sec:GJSConstruction}

\begin{defn}
Given a factor planar algebra $\cP_{\bullet}$, we form the vector space 
$$
\Gr_{k} = \bigoplus_{n=0}^{\infty} \cP_{2k + n}.
$$
Recall from \cite{MR2732052,MR2645882,MR2807103} that $\Gr_{k}$ has the structure of a graded algebra with graded multiplication $\wedge_k$ and normalized Voiculescu trace $\tau_k$ acts on itself by left and right multiplication which is bounded with respect to $\|\cdot \|_2$.
$$
\tau_k(x \wedge_k y)
=
\tau_k\left(
\begin{tikzpicture}[baseline = .1cm]
	\draw (-.8, 0)--(2, 0);
	\draw (0, 0)--(0, .8);
	\draw (1.2, 0)--(1.2, .8);
	\nbox{unshaded}{(0,0)}{.4}{0}{0}{$x$}
	\nbox{unshaded}{(1.2,0)}{.4}{0}{0}{$y$}
	\node at (-.6, .2) {{\scriptsize{$k$}}};
	\node at (.6, .2) {{\scriptsize{$k$}}};
	\node at (1.8, .2) {{\scriptsize{$k$}}};
\end{tikzpicture}
\right)
=
\delta^{-k}
\begin{tikzpicture}[baseline=.3cm]
	\draw (.4,0)--(.4,.6);
	\draw (-.4,0)--(-.4,.6);
	\draw (-.7,0) arc (90:270:.3cm) -- (.7,-.6) arc (-90:90:.3cm) -- (-.7,0);
	\nbox{unshaded}{(-.4,0)}{.3}{0}{0}{$x$}
	\nbox{unshaded}{(.4,0)}{.3}{0}{0}{$y$}
	\nbox{unshaded}{(0,.8)}{.3}{.4}{.4}{$\Sigma \, \TL$}
	\node at (1.1, -.4) {{\scriptsize{$k$}}};		
\end{tikzpicture}
$$
We use the following notation:
\begin{itemize}
\item
$\cA_k$ is the unital $C^*$-algebra generated by the left action of $\Gr_k$ on $L^2(\Gr_k,\tau_k)$,
\item
$\cM_k=\cA_k''=\Gr_k''$ on $L^2(\Gr_k,\tau_k)$.
The algebras $(\cM_k)_{k\geq 0}$ form a Jones tower where
\begin{enumerate}[(i)]
\item
the inclusion $\cM_k \hookrightarrow \cM_{k+1}$ is the extension of the inclusion $\Gr_k\hookrightarrow \Gr_{k+1}$ given by
$$
\begin{tikzpicture}[baseline=-.1cm]
	\draw (-.8, 0)--(.8, 0);
	\draw (0, 0)--(0, .8);
	\nbox{unshaded}{(0,0)}{.4}{0}{0}{$x$}
	\node at (.2, .6) {{\scriptsize{$n$}}};
	\node at (-.6, .2) {{\scriptsize{$k$}}};
	\node at (.6, .2) {{\scriptsize{$k$}}};
\end{tikzpicture}
\longmapsto
\begin{tikzpicture}[baseline=-.1cm]
	\draw (-.8, 0)--(.8, 0);
	\draw (-.8, -.5)--(.8, -.5);
	\draw (0, 0)--(0, .8);
	\nbox{unshaded}{(0,0)}{.4}{0}{0}{$x$}
	\node at (.2, .6) {{\scriptsize{$n$}}};
	\node at (-.6, .2) {{\scriptsize{$k$}}};
	\node at (.6, .2) {{\scriptsize{$k$}}};
\end{tikzpicture}\,,
$$
\item
the Jones projections are given by
$e_k= \delta^{-1} \,
\begin{tikzpicture} [baseline = -.1cm]
	\nbox{}{(0,0)}{.4}{0}{0}{}
	\draw (-.4, .1)--(.4, .1);
	\draw (-.4, 0) arc(90:-90:.15cm);
	\draw (.4, 0) arc(90:270:.15cm);
	\node at (0, .25) {\scriptsize{$k-1$}};
\end{tikzpicture}
\in\Gr_{k+1}$, and
\item
the trace-preserving conditional expectation $\cM_{k+1}\to \cM_k$ is the extension of the conditional expectation $\Gr_{k+1}\to \Gr_{k}$ given by
$$
\begin{tikzpicture}[baseline=-.1cm]
	\draw (-.8, 0)--(.8, 0);
	\draw (0, 0)--(0, .8);
	\nbox{unshaded}{(0,0)}{.4}{0}{0}{$x$}
	\node at (.2, .6) {{\scriptsize{$n$}}};
	\node at (-.6, .2) {{\scriptsize{$k$}}};
	\node at (.6, .2) {{\scriptsize{$k$}}};
\end{tikzpicture}
\mapsto
\delta^{-1}
\begin{tikzpicture}[baseline=-.1cm]
	\draw (-.8, 0)--(.8, 0);
	\draw (0, 0)--(0, .8);
	\draw (-.4,-.2) arc (90:270:.2cm) -- (.4,-.6) arc (-90:90:.2cm);
	\nbox{unshaded}{(0,0)}{.4}{0}{0}{$x$}
	\node at (.2, .6) {{\scriptsize{$n$}}};
	\node at (-.8, .2) {{\scriptsize{$k-1$}}};
	\node at (.8, .2) {{\scriptsize{$k-1$}}};
\end{tikzpicture}.
$$
\end{enumerate}
\item
$\cA_\I$ is the $C^*$-algebra generated by
$$
\Gr_{\I} = \bigoplus_{\ell, n, r = 0}^{\infty}\cP_{\ell, n, r}
$$
where we picture $x \in \cP_{l, n, r}$ as
$
\begin{tikzpicture}[baseline=0cm]
	\draw (-.6, 0)--(.6, 0);
	\draw (0, 0)--(0, .6);
	\nbox{unshaded}{(0,0)}{.3}{0}{0}{$x$}
	\node at (.2, .5) {{\scriptsize{$n$}}};
	\node at (-.5, .2) {{\scriptsize{$l$}}};
	\node at (.5, .2) {{\scriptsize{$r$}}};
\end{tikzpicture}\,
$.
The multiplication on $\cA_{\I}$ is given by
$$
\begin{tikzpicture}[baseline = -.1cm]
	\draw (-.8, 0)--(.8, 0);
	\draw (0, 0)--(0, .8);
	\nbox{unshaded}{(0,0)}{.4}{0}{0}{$x$}
	\node at (.2, .6) {{\scriptsize{$n$}}};
	\node at (-.6, .2) {{\scriptsize{$l$}}};
	\node at (.6, .2) {{\scriptsize{$r$}}};
\end{tikzpicture}
\,\wedge_\I\,
\begin{tikzpicture}[baseline = -.1cm]
	\draw (-.8, 0)--(.8, 0);
	\draw (0, 0)--(0, .8);
	\nbox{unshaded}{(0,0)}{.4}{0}{0}{$y$}
	\node at (.2, .6) {{\scriptsize{$m$}}};
	\node at (-.6, .2) {{\scriptsize{$l'$}}};
	\node at (.6, .2) {{\scriptsize{$r'$}}};
\end{tikzpicture}
\,=\delta_{r,l'}\,
\begin{tikzpicture}[baseline = -.1cm]
	\draw (-.8, 0)--(2, 0);
	\draw (0, 0)--(0, .8);
	\draw (1.2, 0)--(1.2, .8);
	\nbox{unshaded}{(0,0)}{.4}{0}{0}{$x$}
	\node at (.2, .6) {{\scriptsize{$n$}}};
	\nbox{unshaded}{(1.2,0)}{.4}{0}{0}{$y$}
	\node at (1.4, .6) {{\scriptsize{$m$}}};
	\node at (-.6, .2) {{\scriptsize{$l$}}};
	\node at (.6, .2) {{\scriptsize{$r$}}};
	\node at (1.8, .2) {{\scriptsize{$r'$}}};
\end{tikzpicture},
$$
and it carries a non-normalized semi-finite trace
$$
\tau_\I
\left(
\begin{tikzpicture}[baseline = .1cm]
	\draw (-.8, 0)--(.8, 0);
	\draw (0, 0)--(0, .8);
	\nbox{unshaded}{(0,0)}{.4}{0}{0}{$x$}
	\node at (.2, .6) {{\scriptsize{$n$}}};
	\node at (-.6, .2) {{\scriptsize{$l$}}};
	\node at (.6, .2) {{\scriptsize{$r$}}};
\end{tikzpicture}
\right)=
\delta_{l,r}\,
\begin{tikzpicture}[baseline=.5cm]
	\draw (0,0)--(0,.8);
	\nbox{unshaded}{(0,0)}{.4}{0}{0}{$x$}
	\node at (.2, .6) {{\scriptsize{$n$}}};
	\nbox{unshaded}{(0,1.2)}{.4}{.3}{.3}{$\Sigma \, \TL$}
	\draw (-.4,0) arc (90:270:.3cm) -- (.4,-.6) arc (-90:90:.3cm);
	\node at (.8, -.4) {{\scriptsize{$r$}}};		
\end{tikzpicture}\,.
$$
Note that $\cA_k=1_k \cA_\I 1_k$ where $1_k\in \cP_{2k}$ is the diagram with $k$ horizontal strands.
\item
$\cM_\I=\cA_\I''=\Gr_\I''$ on $L^2(\Gr_\I,\tau_\I)$.
\end{itemize}
\end{defn}

\begin{rem}
Note that the restriction of $\tau_\I$ to $\cP_{2k}\subset \Gr_k$ is the usual non-normalized trace on $\cP_{2k}$. We denote this restriction by $\Tr$.
\end{rem}

\subsection{The filtered convention}

$\Gr_{\I}$ also comes equipped with a filtered algebra structure.  The multiplication $\cdot$ is given by:
$$ x \cdot y = \delta_{r, l'} \cdot
\sum_{j = 0}^{\min\{n, m\}} \begin{tikzpicture}[baseline=.2cm]
	\draw(-1.3, 0)--(1.3, 0);
	\draw (-.7,0)--(-.7,1);
	\draw (.7,0)--(.7,1);
	\draw (-.3,.4) arc (180:0:.3cm);
	\node at (-1.1, .2) {\scriptsize{$l$}};
	\node at (1.1, .2) {\scriptsize{$r'$}};
	\node at (-1.1,1.2) {\scriptsize{$n-j$}};
	\node at (-.3,.7) {\scriptsize{$j$}};
	\node at (1.2,1.2) {\scriptsize{$m-j$}};
	\nbox{unshaded}{(-.5,0)}{.4}{0}{0}{$x$}
	\nbox{unshaded}{(.5,0)}{.4}{0}{0}{$y$}
\end{tikzpicture}
$$
for $x \in \cP_{l, n, r}$ and $y \in \cP_{l', m, r'}$.  There is a trace $\Tr$ on $\Gr_{\infty}$ given by
$$
\Tr(x) = \delta_{n, 0} \cdot \delta_{l, r}\,
\begin{tikzpicture}[baseline = -.2cm]
	\draw (-.4, 0)--(.4, 0) arc(90:-90:.4cm)--(-.4, -.8) arc(270:90:.4cm);
	\nbox{unshaded}{(0,0)}{.4}{0}{0}{$x$}
	\node at (-.6, .2) {{\scriptsize{$l$}}};
	\node at (.6, .2) {{\scriptsize{$r$}}};
\end{tikzpicture}
$$   
for $x \in \cP_{l, n, r}$.  Consider the map $\Phi: \Gr_{\I} \rightarrow \Gr_{\I}$ given by:
$$
\Phi(x) = \sum_{T \in \Epi(\TL_{\bullet})} \begin{tikzpicture}[baseline = .4cm]
	\draw (-.8, 0)--(.8,0);
	\draw (0, 0)--(0,.8);
	\nbox{unshaded}{(0,0)}{.4}{0}{0}{$x$}
	\nbox{unshaded}{(0,1)}{.4}{0}{0}{$T$}
\end{tikzpicture}
$$ 
where $\Epi(\TL_{\bullet})$ is the set of Temperley-Lieb diagrams with marked points on the top and bottom so that every string which has an endpoint on the top has its other endpoint on the bottom.  Note that $\Phi$ preserves elements in $\cP_{l, 0, r}$ and $\cP_{l, 1, r}$. 
From the arguments in \cite[Section 5]{MR2645882}, $\Phi$ is a bijection with the property that
$$
\Phi(x \wedge_{\I} y) = \Phi(x)\cdot \Phi(y) \text{ and } \tau_{\infty}(x) = \Tr(\Phi(x)).
$$
This shows that $\cA_{\I}$ is isomorphic to the $C^{*}$-algebra generated by $\Gr_{\I}$ acting on $L^{2}(\Gr_{\I}, \Tr)$ via the multiplication $\cdot$.  

We now want to identify $\cA_\infty$ with $\SP{}$. 
To do so, we need a few definitions.
First, let $L^{2}(\cB, \Tr)$ be the closure of $\set{x \in \cB}{\Tr(x^{*}x) < \infty}$ with respect to the norm $\|\cdot\|_{2}$ where $\|x\|_{2}^{2} = \Tr(x^{*}x)$.
Notice that $\cB$ has an obvious faithful left action on $L^{2}(\cB, \Tr)$.  

Next, let $\cH_\infty=\FP{} \otimes_{\cB} L^{2}(\cB, \Tr)$, and note that $\cH_\infty$ has a left $\SP{}$-action.
Since $\cB$ acts on the right of $\FP{}$ faithfully, the induced representation of $\cL(\FP{}))$ on $\cH_{\I}$ is faithful, so the left $\SP{}$-action is faithful.  

Finally, recall from \cite[Section 5.5]{CStarFromPAs} that there is a conditional expectation $E_\infty: \SP{} \rightarrow \cB$ given by the extension of $E_\infty(x) = \delta_{n, 0}x$ for $x \in \cP_{l, n, r}\subset \SP{}$.
Moreover, the GNS Hilbert space $L^{2}(\SP{}, \Tr\circ E_\infty)$ can be naturally identified with $L^2(\Gr_\infty,\Tr)$. 

\begin{lem}\label{lem:AInftyIso}
There is a unitary $V$ from $\cH_\infty$ to the GNS Hilbert space $L^{2}(\Gr_\infty,\Tr)$ intertwining the left $\SP{}$ and left $\cA_\infty$ actions.
Hence $\cA_\infty \cong \SP{}$.
\end{lem}
\begin{proof}
We define $V$ on elementary tensors of the form $x\otimes a$ where $x\in \cP_{l,n,r}\subset \cX_n$ and $a\in \cP_{r,s}\subset \cB$ by
$$
x\otimes a \mapsto
\begin{tikzpicture}[baseline=0cm]
	\draw (0,0)--(0,.8);
	\draw (-.8,0)--(1.8,0);
	\node at (-.2,.6) {\scriptsize{$n$}};
	\node at (-.6,.2) {\scriptsize{$l$}};
	\node at (.6,.2) {\scriptsize{$r$}};
	\node at (1.6,.2) {\scriptsize{$s$}};
	\nbox{unshaded}{(0,0)}{.4}{0}{0}{$x$}
	\nbox{unshaded}{(1.1,0)}{.3}{0}{0}{$a$}
\end{tikzpicture}\,.
$$
For $y\in \cP_{l',n',r'}\subset \cX_{n'}$ and $b\in \cP_{r',s'}\subset \cB$, we calculate that
$$
\langle x\otimes a | y\otimes b\rangle_\bbC 
= \langle a | \langle x | y\rangle_\cB b\rangle_{L^2(\cB,\Tr)}
=
\delta_{l,l'}\delta_{n,n'}\delta_{s,s'}\,
\begin{tikzpicture}[baseline=-.1cm]
	\draw (-1,.4) arc (180:90:.2cm) -- (-.2,.6) arc (90:0:.2cm);
	\draw (-2.2,0)--(1.2,0) arc (90:-90:.3cm) -- (-2.2,-.6) arc (270:90:.3cm);
	\node at (-.5,.8) {\scriptsize{$n$}};
	\node at (1.4,.1) {\scriptsize{$s$}};
	\nbox{unshaded}{(-1,0)}{.4}{0}{0}{$x^*$}
	\nbox{unshaded}{(0,0)}{.4}{0}{0}{$y$}
	\nbox{unshaded}{(.9,0)}{.3}{0}{0}{$b$}
	\nbox{unshaded}{(-1.9,0)}{.3}{0}{0}{$a^*$}
\end{tikzpicture}\,.
$$
The rest is straightforward.
\end{proof}

Since the map $\Phi$ is the identity on $\cB$, we have the following corollary:

\begin{cor}\label{cor:ZerothGJS}
$1_{k}\SP{}1_{k} \cong \cA_{k}$. 
\end{cor}

For the rest of this article, we identify $\cA_{\I}$ with  $\SP{}$, and we always write $\cA_{\I}$.

\subsection{The shaded case}
When $\cP_{\bullet}$ is a subfactor planar algebra, then as in \cite{MR2732052,MR2645882}, we define
$$
\Gr_{k,\pm} = \bigoplus_{n=0}^{\infty} \cP_{2k + n,\pm (-1)^k}.
$$
As before, we picture elements in $\Gr_{k,\pm}$ as boxes of the form
$
\begin{tikzpicture}[baseline=-.1cm]
	\draw (-.6, 0)--(.6, 0);
	\draw (0, 0)--(0, .6);
	\nbox{unshaded}{(0,0)}{.3}{0}{0}{$x$}
	\node at (.2, .6) {{\scriptsize{$n$}}};
	\node at (-.5, .2) {{\scriptsize{$k$}}};
	\node at (.5, .2) {{\scriptsize{$k$}}};
\end{tikzpicture}
$
with the marked region on the bottom of the box. In this manner, the upper left corner will be shaded $\pm$ and the shading in the marked region depends on the parity of $k$.

As in the first two sections, we can complete these to form $C^{*}$-algebras $\cA_{k, \pm}$ and von Neumann algebras $\cM_{k, \pm}$.  As was shown in \cite{MR2732052}, $\cM_{0, \pm}' \cap \cM_{k, \pm} \cong \cP_{2k, \pm}$.  

We also form semifinite algebras $\Gr_{\infty,\pm}$ which are given by 
$$
\Gr_{\I, \pm} = \bigoplus_{l, n, r = 0}^{\infty} \cP_{l, n, r, \pm}
$$
where the marked region is on the bottom of the box and is always $\pm$.  
Just as in the previous two sections, $\Gr_{\I,\pm}$ completes to $C^{*}$-algebras $\cA_{\I, \pm}$ and $\cM_{\I, \pm}$ using either of the two multiplications and traces introduced there.  
If $1_{k, \pm}$ is the element in $\cP_{2k,\pm}$ with $k$ through strings, then we have the following:
\begin{itemize}
\item 
$1_{k, \pm}\cA_{\I, \pm}1_{k, \pm} \cong
\begin{cases}
\cA_{k,\pm} & \text{for $k$ even}
\\
\cA_{k,\mp} & \text{for $k$ odd.}
\end{cases}
$
\item 
$\cS(\cP_{\bullet, \pm}) \cong \cA_{\I,\pm}$.

\end{itemize}
 
Throughout the rest of this article, we assume that $\cP_{\bullet}$ is unshaded or that the shading is implicit.  
If we need a specific shading for an example, it will be explicitly mentioned.

%% file: Chapters/RieffelMorita.tex

\section{Strong Morita equivalence}\label{sec:Morita}

The main result of this section is showing that $\cA_{0}$ is Morita equivalent to $\cS(\Gamma)$ from Subsection \ref{sec:compression}.  This will show that $A_{0}$ is strongly Morita equivalent to $\cA_{k}$ for all $k$ and that $\cA_{\I} \cong \cA_{0} \otimes \cK(H)$.  The computation of the $K$-groups of the $\cA_{k}$ will follow as a result.

After recalling the basic definitions and properties of strong Morita equivalence, we will exhibit a simple planar algebraic proof of Morita equivalence in finite depth.  We will then collect some facts about reduced free products which will allow us to show our desired Morita equivalence computations.

\subsection{Strong Morita equivalence}

To begin, we recall the basic notions of strong Morita equivalence. 
In \cite{MR0353003,MR0367670}, Rieffel defined the notion of strong Morita equivalence for $C^*$-algebras via equivalence bimodules. 
We will use an equivalent definition below. 
To simplify our definitions, we will assume all our $C^*$-algebras are separable.

\begin{defn}\label{defn:StrongRME}
Two $C^*$-algebras $A,B$ are \underline{strongly Morita equivalent}, denoted $A\sim B$,  if there is a $C^*$-algebra $C$ such that  $A,B$ are full hereditary subalgebras of $C$.
Recall that $A\subset C$ is full if $CAC$ is dense in $C$ and hereditary if $a\in A^+$ and $0\leq c\leq a$ implies $c\in A$.
\end{defn}

We also use the following facts.

\begin{facts}\label{thm:StableIsoME}
\mbox{}
\be
\item
$A\sim B$ if and only if $A,B$ are stably isomorphic, i.e. $A\otimes \cK \cong B\otimes \cK$ \cite[Theorem 1.2]{MR0463928}.
\item
If $A\sim B$, then $K_*(A)\cong K_*(B)$.
\ee
\end{facts}

Recall that a $C^*$-algebra is simple if it has no closed 2-sided ideals.
The following proposition is well known, but we include a short proof for the reader's convenience.

\begin{prop}\label{prop:Simple}
Suppose $A\sim B$ via $C$ as in Definition \ref{defn:StrongRME}. Then simplicity of $A$, $B$, or $C$ implies the simplicity of all three.
\end{prop}
\begin{proof}
If $A$ is simple, then $A\otimes \cK$ is simple. Since $A\otimes \cK\cong C\otimes \cK$, we get that $C\otimes \cK$ is simple, and thus so is $C$.  The other cases have similar arguments.
\end{proof}

\subsection{Some easy consequences}

In this subsection, we use the graded convention, in which the cup elements $\cup_{j, k}$ in Fact \ref{fact:CupInvertible} are invertible.

\begin{lem}\label{lem:StrongRME}
For $i<j$, $\cA_i\sim \cA_j$ if $\cA_j=1_j\cA_{\I}1_i \cA_{\I} 1_j$.
\end{lem}
\begin{proof}
Set $A=\cA_i$, $B=\cA_j$, and $C=(1_i+1_j)\cA_\I (1_i+1_j)$ as in Definition \ref{defn:StrongRME}.
It is clear that $A=1_iC 1_i$, so $A\sim B$ if $\cA_j=1_j C 1_j$. 
The latter is true if $\cA_j=1_j\cA_{\I}1_i \cA_{\I} 1_j$.
\end{proof}

\begin{fact}[{\cite[Lemma 2]{MR2732052}, \cite[Lemma 3.17]{1208.5505}}]\label{fact:CupInvertible}
Since $\delta>1$, the cabled cup element
$
\cup_{j,k}=
\begin{tikzpicture} [baseline = -.1cm]
 	\draw  (-.4, -.2) -- (.4, -.2);
	\draw (-.25, .4) arc(-180:0: .25cm);
	\node at (-.29, .19) {{\scriptsize{$j$}}};
	\node at (.25, -.05) {{\scriptsize{$k$}}};
	\nbox{}{(0,0)}{.4}{0}{0}{}
\end{tikzpicture}
$
is positive and invertible in $\cM_k$, so it is invertible in $\cA_k$ by spectral permanence.
\end{fact}

\begin{defn}
We denote the inverse of $\cup_{j,k}$ by the formal symbol
\begin{tikzpicture} [baseline = -.1cm]
	\draw (0,0)--(0,.7);
 	\draw  (-.6, 0) -- (.6, 0);
	\nbox{unshaded}{(0,0)}{.4}{0}{0}{$\cup_{j,k}^{-1}$}
\end{tikzpicture}\,,
even though it is not in $\Gr_k$.
\end{defn}

We can now show that for finite depth $\cP_\bullet$, eventually the GJS $C^*$-algebras are all strongly Morita equivalent.

\begin{thm}\label{thm:FiniteDepth}
Suppose $\cP_\bullet$ has depth $n<\I$. Then $\cA_{n-1}\sim \cA_{k}$ for all $k\geq n-1$.
\end{thm}
\begin{proof}
Since $\cP_\bullet$ has depth $n$, we know that $\cP_{n}e_{n} \cP_{n}=\cP_{n+1}$, and hence $n+1$ through strings factors through $n-1$ through strings. This means there is a finite Pimsner-Popa basis $\{b\}\subset \cP_n$ such that $\sum_b be_nb^*=1_{n+1}$. In diagrams, we have
$$
\begin{tikzpicture}[baseline = -.1cm]
	\draw (-.8, 0)--(.8, 0);
	\node at (0, -.2) {{\scriptsize{$n+1$}}};
\end{tikzpicture}
=
\begin{tikzpicture}[baseline = -.1cm]
	\draw (-1.4, 0)--(2.2, 0);
	\draw (0,.2) arc (-90:90:.2cm) -- (-1.4,.6);
	\draw (1,.2) arc (270:90:.2cm) -- (2.2,.6);
	\nbox{unshaded}{(-.4,0)}{.4}{0}{0}{$b$}
	\nbox{unshaded}{(1.4,0)}{.4}{0}{0}{$b^*$}
	\node at (-1.2, -.2) {{\scriptsize{$n$}}};
	\node at (.5, -.2) {{\scriptsize{$n-1$}}};
	\node at (2, -.2) {{\scriptsize{$n$}}};
\end{tikzpicture}\,,
$$
and we immediately see that $1_{n+1}\cA_{\I}1_{n-1} \cA_{\I} 1_{n+1}=\cA_{n+1}$, and thus $\cA_{n-1}\sim \cA_{n+1}$ by Lemma \ref{lem:StrongRME}.
Also, $\cA_k\sim \cA_{k+2}$ for $k\geq n$ follows similarly since $k+2$ strings factors through $k$ strings by taking a Pimsner-Popa basis for $\cP_{k+1}$ over $\cP_k$.

Finally, we show $\cA_{n-1}\sim \cA_n$. 
As above, let $\{b\}\subset \cP_n$ be a Pimsner-Popa basis for $\cP_n$ over $\cP_{n-1}$. Then we have
$$
\begin{tikzpicture}[baseline = -.1cm]
	\draw (-.8, 0)--(.8, 0);
	\node at (0, -.2) {{\scriptsize{$n$}}};
\end{tikzpicture}
=
\begin{tikzpicture}[baseline = -.1cm]
	\draw (-1, -.1)--(1.2, -.1);
	\nbox{}{(-.4,0)}{.4}{0}{0}{}
	\node at (-.4, -.25) {{\scriptsize{$n$}}};
	\draw (-.6,.8) -- (-.6,.4) arc (-180:0:.2cm) -- (-.2,.8);
	\draw (.6, 0)--(.6, .8);
	\nbox{unshaded}{(.6,0)}{.4}{0}{0}{$\cup_{1,n}^{-1}$}
\end{tikzpicture}
=
\begin{tikzpicture}[baseline = -.1cm]
	\draw (-1.4, 0)--(3.4, 0);
	\draw (0,.2) arc (-90:90:.2cm) -- (-.8,.6) arc (270:180:.2cm);
	\draw (1,.2) arc (270:90:.2cm) -- (1.8,.6) arc (-90:0:.2cm);
	\nbox{unshaded}{(-.4,0)}{.4}{0}{0}{$b$}
	\nbox{unshaded}{(1.4,0)}{.4}{0}{0}{$b^*$}
	\node at (-1.2, -.2) {{\scriptsize{$n$}}};
	\node at (.5, -.2) {{\scriptsize{$n-1$}}};
	\node at (2, -.2) {{\scriptsize{$n$}}};
	\draw (2.6, 0)--(2.6, .8);
	\nbox{unshaded}{(2.6,0)}{.4}{0}{0}{$\cup_{1,n}^{-1}$}
	\node at (3.2, -.2) {{\scriptsize{$n$}}};
\end{tikzpicture}\,.
$$
Hence $1_n \cA_\I 1_{n-1} \cA_\I 1_n=\cA_n$, and again $\cA_{n-1}\sim \cA_n$ by Lemma \ref{lem:StrongRME}.
\end{proof}

\begin{rem}
If $\cP_\bullet$ is shaded, then this argument shows that $\cA_{k,\pm}\sim A_{k\pm 1, \mp}$ for all $k\geq \depth(\cP_\bullet)-1$.
However, in general, it is not the case that $A_{k,\pm}\sim A_{k+1,\pm}$ (see Example \ref{ex:NotME}).
\end{rem}

\subsection{Some useful free product theorems}\label{sec:freeproduct}
The upcoming Morita equivalence arguments, as well as the arguments for simplicity, unique trace, and the positive cone of $K_{0}(\cA_{0})$, rely on some results about reduced free products.  
The first of these is due to Avitzour.
We remind the reader that all free products in this article are reduced.

\begin{thm}[\cite{MR654842}] \label{thm:Avitzour}
Let $(B, \tr) =  (B_{1}, \tr_{1}) \Asterisk (B_{2}, \tr_{2})$ where $\tr_{i}$ is a faithful tracial state on $B_{i}$ for $i \in \{1, 2\}$, and $\tr$ is the canonical free product tracial state.  Suppose that there exist unitaries $u_{1} \in B_{1}$, and $u_{2}, u_{2}' \in B_{2}$ satisfying
$$
\tr(u_{1}) = 0 = \tr(u_{2}) = \tr(u_{2}') = \tr(u_{2}^{*}u_{2}').
$$
Then $B$ is simple, and $\tr$ is the unique tracial state on $B$.
\end{thm}

If $B_{1}$ and $B_{2}$ contain such unitaries, we say that $B$ satisfies the \underline{Avitzour condition}.  
A unital $C^{*}$-algebra $B$ has \underline{stable rank 1} if the set of invertible elements is dense in $B$.   
Dykema, Haagerup, and R{\o}rdam showed that certain reduced free products have stable rank 1.

\begin{thm}[\cite{MR1478545}] \label{thm:AvitzourSR}
Suppose $(B, \tr) =  (B_{1}, \tr_{1}) \Asterisk (B_{2}, \tr_{2})$ satisfies the Avitzour condition as in Theorem \ref{thm:Avitzour}.  Then $B$ has stable rank 1.
\end{thm}

Recall that for a $C^{*}$-algebra $B$, the positive cone $K_{0}(B)^{+}$ of $K_{0}(B)$ is the set
$$
K_{0}(B)^{+} = \{x \in K_{0}(B): x = [p] \text{ for } p \in P(M_{n}(B))\}.
$$
Dykema and R{\o}rdam observed that $K_{0}(B)^{+}$ can be quite large for $B$ a free product.

\begin{thm}[\cite{MR1601917}]\label{thm:positivecone}
Suppose $(B, \tr) =  (B_{1}, \tr_{1}) \Asterisk (B_{2}, \tr_{2})$ satisfies the Avitzour condition as in Theorem \ref{thm:Avitzour}.  Let $j_{i}$ be the canonical inclusion of $B_{i}\hookrightarrow B$ and $K_{0}(j_{i})$ the induced map on $K$-theory.  Let $G$ be the subgroup $K_{0}(j_{1})(K_{0}(B_{1})) + K_{0}(j_{2})(K_{0}(B_{2}))$ of $K_{0}(B)$.  Then
$$
K_{0}(B)^{+} \cap G = \{x \in G: \tr(x) > 0\} \cup \{0\}.
$$
\end{thm}

A tracial $C^{*}$-algebra $(B, \tr)$ contains a diffuse abelian $C^{*}$-subalgebra $\cC = C(X)$ if the Borel measure on $X$ induced from $\tr$ contains no atoms. 
Dykema was able to recover the results of Theorems \ref{thm:Avitzour} and \ref{thm:AvitzourSR} under a slightly different condition on the free product.
\begin{thm}[\cite{MR1473439}] \label{thm:diffuse}
\mbox{}
\be
\item 
Suppose $(B, \tr) =  (B_{1}, \tr_{1}) \Asterisk (B_{2}, \tr_{2})$ where $(B_{1}, \tr_{1})$ contains a diffuse abelian $C^{*}$-subalgebra and $B_{2} \neq \C$.  
Then $B$ is simple, has stable rank 1, and has unique tracial state $\tr$.

\item 
Suppose $(B, \tr) = (B_{1}, \tr_{1}) \Asterisk \cdots \Asterisk (B_{n}, \tr_{n})$, where each $B_{i}$ has the form
$$
B_{i} = \underset{\mu_{i}}{\cD} \oplus \C
$$
where $\cD$ contains a diffuse abelian subalgebra.  
Then $B$ is simple and has unique trace if and only if $\sum_{i=1}^{n}\mu_{i} > \tr(1)$.  
$B$ always has stable rank 1, regardless of the weighting.
\ee
\end{thm}
Finally, we finish this section with a technical tool for compressions of free products.  
Also see \cite[Theorem 1.2]{MR1201693} for a von Neumann algebra version.
\begin{lem}[\cite{MR1201693,MR1473439}] \label{lem:compressfree}
Suppose that $A$, $B$, and $C$ are tracial $C^{*}$-algebras with
$$
\mathcal{D} = (\overset{p}{A} \oplus B) {\Asterisk} C,
$$
and $\mathcal{D}$ is endowed with the canonical free product trace.  Then
$
p\mathcal{D}p = A {\Asterisk} p\left((\overset{p}{\C} \oplus B) {\Asterisk} C\right)p.
$
\end{lem}
\subsection{Morita equivalence of $\cA_{0}$ and $\cS(\Gamma)$}\label{sec:MoritaGJS}

We will now prove the following theorem.

\begin{thm}\label{thm:RMEGraph}
$\cA_{0} \sim \cS(\Gamma)$.
\end{thm}

\begin{proof}
We note that $\cA_{0} = p_{\star}\cS(\Gamma)p_{\star}$ for $\mu$ the quantum dimension weighting, so we need to show that $p_{\star}$ is full in $\cS(\Gamma)$.  Let $\alpha$ be a vertex connected to $\star$ by an edge, $\e$, and let $\e' \in E(\vec{\Gamma})$ have $s(\e') = \star$ and $t(\e') = \alpha$ so that $g_{\e'}^{*}g_{\e'} = g_{\e'}^{*}p_{\star}g_{\e'}$ is supported under $p_{\alpha}$.  Let $\e_{1}, \dots, \e_{n}$ be the collection of edges having target $\alpha$.  Then by Facts \ref{facts:edge}, $p_{\alpha}\cS(\Gamma)p_{\alpha}$ contains the following free product
$$
\underset{i=1}{\overset{n}{\Asterisk}} \underset{\min\{\mu(s(\e_{i})), \mu(\alpha)\}}{C[0, 1]} \oplus \C
$$
where the trace on $C[0, 1]$ is integration against Lebesgue measure.  
The Frobenius-Perron condition implies
$$
\sum_{i=1}^{n} \min\{\mu(s(\e_{i})), \mu(\alpha)\} \geq \mu(\alpha)
$$
so that this free product is simple by Theorem \ref{thm:diffuse}.  This shows that $p_{\alpha}$ is in the ideal generated by $p_{\star}$.  Continuing inductively shows that $p_\beta$ is in the ideal generated by $\star$ for every $\beta\in V(\Gamma)$.
\end{proof}

From this, we deduce the following three corollaries.

\begin{cor}
$\cA_{\I} \cong \cA_{0} \otimes \cK$ for $\cK$ the algebra of compact operators on a separable, infinite dimensional Hilbert space.
\end{cor}

\begin{proof}
The isomorphism $\cA_{\I} \cong \cS(\Gamma) \otimes \cK$ was proven in \cite{CStarFromPAs}.  
The corollary is then immediate from Facts \ref{thm:StableIsoME} and Theorem \ref{thm:RMEGraph}. 
\end{proof}

\begin{cor}\label{cor:allME}
$\cA_{0} \sim \cA_{k}$ for all $k \geq 0$.
\end{cor}

We provide two proofs of this corollary.

\begin{proof}[Proof 1 of Corollary \ref{cor:allME}]
Note that $1_{k}$ is a linear combination of projections equivalent to $p_{\alpha}$ for $\alpha \in V(\Gamma(k))$.  
The result is now immediate from Theorem \ref{thm:RMEGraph}.
\end{proof}

\begin{cor}
$K_{0}(\cA_{k}) = \Z\set{[p_{\alpha}]}{\alpha \in V(\Gamma)}$ and $K_{1}(\cA_{k}) =\{0\}$ for all $k \geq 0$.
\end{cor}

\begin{rem} 
The first corollary is analogous to the fact that $\cM_{\I} \cong \cM_{0} \otimes \cB(H)$ for $H$ separable and infinite dimensional \cite{MR2807103}.
If $\cP_{\bullet, \pm}$ is a subfactor planar algebra, the second and third corollaries need to to be modified, since $\cA_{k, \pm}$ is a corner of $\cA_{\I, \mp}$ for $k$ odd.
In this case, we have $\cA_{k, \pm} \sim \cA_{k+1, \mp}$ for all $k$.  
For the $K$-groups, we have the following:
\begin{align*}
K_{0}(\cA_{k, \pm}) &= 
\begin{cases}
\Z\set{[p_{\alpha}]}{\alpha \in V(\Gamma_{\pm})} &\text{ for } k \text { even }
\\
\Z\set{[p_{\alpha}]}{\alpha \in V(\Gamma_{\mp})} &\text{ for } k \text { odd }
\end{cases}
K_{1}(\cA_{k, \pm}) &= \{0\} \text{ for all } k \geq 0
\end{align*}
\end{rem}

While the following is not needed for the rest of the article, it is interesting to note that it is actually sufficient to prove Corollary \ref{cor:allME} only using the $A_{n}$ Coxeter-Dynkin diagram ($n\in \{3, 4, \dots\} \cup \{\I\}$), which is the principal graph for Temperley-Lieb $\TL_\bullet(\delta)$ for $\delta\in \set{2\cos(\pi/(k)}{k\geq 3}\cup[2,\I)$.  
Indeed, since $\TL_{\bullet}$ sits inside every factor planar algebra, we have $\cS(A_{n}) \subset \cA_{\I}$.  
Let $\jw{k}$ be the $k^{\text{th}}$ Jones Wenzl idempotent in $\TL_{2k}$, so it is a representative of the depth $k$ vertex of $A_{n}$.  
We also set $g_{k} = g_{\e}$ if $\e$ has $s(\e) = k-1$ and $t(\e) = k$.

\begin{proof}[Proof 2 of Corollary \ref{cor:allME}]

It is sufficient to show that for all $k$, $\jw{k}$ is in the ideal (in $\cA_{\I}$) generated by $p_{\star} = \jw{0}$ since $1_{k}$ is a linear combination of projections equivalent to various $\jw{j}$ for $j \leq k$.  
Note that the element
$$
g_{k-1}\cdots g_{0}g_{0}^{*} \cdots g_{k-1}^{*} 
= 
g_{k-1}\cdots g_{0}p_{\star}g_{0}^{*} \cdots g_{k-1}^{*}
=
(\text{constant})\cdot\,
\begin{tikzpicture}[baseline = -.1cm]
	\draw (-1.6, 0)--(-.8, 0);
	\draw (1.6, 0)--(.8, 0);
	\draw (-.4,0) arc (-90:0:.2cm) -- (-.2,.7);
	\draw (.4,0) arc (270:180:.2cm) -- (.2,.7);
	\nbox{unshaded}{(-.8,0)}{.4}{0}{0}{$\jw{k}$}
	\nbox{unshaded}{(.8,0)}{.4}{0}{0}{$\jw{k}$}
	\node at (-.4, .6) {{\scriptsize{$k$}}};
	\node at (.4, .6) {{\scriptsize{$k$}}};
	\node at (-1.4, .2) {{\scriptsize{$k$}}};
	\node at (1.4, .2) {{\scriptsize{$k$}}};
\end{tikzpicture}
$$
is nonzero, in the ideal generated by $p_{\star}$, and in $\jw{k}\cS(A_{n})\jw{k} \subset \jw{k}\cA_{\I}\jw{k}$.  
The algebra $\jw{k}\cS(A_{n})\jw{k}$ is simple by the same arguments as in the proof of Theorem \ref{thm:RMEGraph}.  
The result follows.
\end{proof}

%% file: Chapters/GJSProperties.tex

\section{Properties of the GJS $C^*$-algebras}\label{sec:properties}

We now turn our attention to deriving many properties about the $C^{*}$-algebras $\cA_{k}$.  
Specifically, we  show that these algebras are simple with unique tracial state.  
Furthermore we will show that the $K_{0}$ group of each $\cA_{k}$ is weakly unperforated as an ordered abelian group and that each $\cA_{k}$ has stable rank 1.  
As a consequence, we show that $\cA_{\I}$ has a comparability of projections and that a strong dichotomy occurs regarding projections in $\cA_{0}$.  
Either $\cA_{0}$ is projectionless, or $\set{\tr(p)}{p \in P(\cA_{0})}$ is dense in $[0, 1]$.

Along the way, we observe that the known results for free products over the scalars are insufficient to help derive these results for a (sub)factor planar algebra $\cP_{\bullet}$.  
We therefore develop analogous theorems for amalgamation over $\C^{2}$ which prove useful in our analysis.
\subsection{Simplicity, unique trace, and weak unperforation}\label{sec:Simple}

We now show that the algebras $\cA_{k}$ are simple with unique trace, and their $K_{0}$ group is weakly unperforated, i.e. $nx > 0$ for some $n \in \N$ implies $x > 0$.  
It turns out that the analysis simplifies considerably if $\Gamma$ has a particular decomposition.
This is where we begin.

\begin{thm}\label{thm:SimpleTrace1}
Suppose that one of the following conditions hold:
\be
\item
$\Gamma$ can be written as the union $\Gamma_1\cup \Gamma_2$, where $\Gamma_1$ and $\Gamma_2$ each have at least one edge, and the intersection $\Gamma_1\cap \Gamma_2$ only consists of the vertex $\star$.
\item
$\cP_\bullet$ is irreducible, i.e., $\Gamma$ has only one vertex at depth zero, only one vertex at depth 1, and only one edge between them.
\ee
Then $\cA_0$ is simple with unique tracial state, and
$
K_{0}(\cA_{0})^{+} = \{x \in K_{0}(\cA_{0}): \Tr(x) > 0\} \cup \{0\}.
$
\end{thm}
\begin{proof}
Throughout this proof, we assume that $\mu$ is the quantum dimension weighting on $V(\Gamma)$ coming from $\cP_{\bullet}$. 
If (1) holds, then
$$
\cA_{0} \cong p_{\star}\cS(\Gamma_{1})p_{\star} {\Asterisk}\, p_{\star}\cS(\Gamma_{2})p_{\star}.
$$
Let $\e_{1} \in E(\vec{\Gamma}_{1})$ and $\e_{2} \in E(\vec{\Gamma}_{2})$ with $t(\e_{i}) = \star$.
Then $g_{\e_{i}}^{*}g_{\e_{i}}$ generates a diffuse abelian $C^{*}$ algebra for $i \in \{1, 2\}$, so it follows that $\cA_{0}$ satisfies the Avitzour condition.  
Furthermore, the proof of Theorem \ref{thm:RMEGraph} implies $p_{\star}$ is full in both $\cS(\Gamma_{1})$ and $\cS(\Gamma_{2})$.  
Therefore, if $\alpha \in V(\Gamma)$, then $[p_{\alpha}]$ is in the $K_{0}$ group of at least one of the $C^{*}$ algebras in the free product.  
Thus $\cA_{0}$ is simple, has unique tracial state, and the positive cone of $K_{0}(\cA_{0})$ has the desired properties.

If (2) holds, let $\e$ denote the (unique) edge in $\vec{\Gamma}$ with $t(\e) = \star$ and $s(\e) = \alpha$, the single depth 1 vertex. Note that $p_{\alpha} = \jw{1}$ and $\Tr(\jw{1}) = \delta > 1$.  Therefore, $g_{\e}^{*}g_{\e}= (\text{constant})\cdot \cup$ is invertible in $\cA_{0}$, so both elements on the right side of the polar decomposition $g_{\e} = w_{\e}|g_{\e}|$ are in $\cS(\Gamma, \mu)$.  Set
$$
\cA_{\e} = w_{\e}C^{*}(g_{\e}^{*}g_{\e})w_{\e}^{*} \text{ and } \cA_{\Gamma \setminus \{\e\}} = \jw{1}C^{*}(\Gamma \setminus \{\e\})\jw{1}.
$$
Then from Lemma \ref{lem:compressfree}, we see that
$$
\cA_{0} 
\cong w_{\e}\cA_{\I}w_{\e}^{*} 
= 
\cA_{\e} {\Asterisk} w_{\e}w_{\e}^{*}\left( \left(\overset{w_{\e}w_{\e}^{*}}{\C} \oplus \C\right) {\Asterisk} \cA_{\Gamma \setminus \{\e\}}\right)w_{\e}w_{\e}^{*}
$$
The algebra $\cA_{\e}$ is diffuse abelian, and from \cite{MR3110503}, the second algebra in the free product completes to an interpolated free group factor, so it contains a unitary of trace zero.  Therefore $\cA_{0}$ satisfies the Avitzour condition.  Furthermore, the arguments in Theorem \ref{thm:RMEGraph} imply that $w_{\e}w_{\e}^{*}$ is full in $\cA_{\Gamma \setminus \{\e\}}$ so we can deduce that each $[p_{\beta}]$ for $\beta$ a vertex of $\Gamma$ appears in the $K_{0}$ group of at least one of the two algebras in the free product.  Therefore, we again deduce that $\cA_{0}$ is simple, has unique tracial state, and the positive cone of $K_{0}(\cA_{0})$ has the desired properties.
\end{proof}

We now suppose that our planar algebra $\cP_\bullet$ does not satisfy either of the conditions in Theorem \ref{thm:SimpleTrace1}.  
Then the principal graph $\Gamma$ must satisfy the following condition.

\begin{assumption}\label{assume:Cycle}
There is an edge $\e$ in $\Gamma$ connecting $\star$ to a vertex $\alpha \neq \star$, and there is a path from $\star$ to $\alpha$ which avoids $\e$.
\end{assumption}

We now show the algebra $\cA_0$ is simple with unique trace with the above assumption, even though we are no longer able to express $\cA_{0}$ as a free product over the scalars.  In order to analyze this case we will need some additional free product lemmas which are in the spirit of Subsection \ref{sec:freeproduct} for amalgamation over $\C^{2}$.

We start with the following lemma, which is an analogue of Lemma \ref{lem:compressfree} for amalgamated free products, and the techniques used in the proof mirror the techniques used in \cite{MR1201693}.

\begin{lem} \label{lem:compressfree2}
Suppose there are two unital, tracial, $C^{*}$-algebras $\overset{p + r}{B} \oplus \overset{r'}{\C} \text{ and } C$ which both contain $D = \overset{p}{\C} \oplus \overset{q}{\C}$ as a unital $C^{*}$-subalgebra with $q = r + r'$. Assume that the algebras are equipped with conditional expectations $E^{1}_{D}$ and $E^{2}_{D}$ onto $D$ respectively as well as traces $\tr_{1}$ and $\tr_{2}$ so that $\tr_{i} = \tr_{i} \circ E^{i}_{D}$ for $i = 1, 2$, and the restrictions of $\tr_{1}$ and $\tr_{2}$ to $D$ coincide.  Form the reduced amalgamated free product
$$
A = \left( \overset{p + r}{B} \oplus \overset{r'}{\C} \right) \underset{D}{\Asterisk}\, C.
$$
Then
$$
(p+r)A(p+r) = (B, E^{1}_{D'}) \underset{D'}{\Asterisk}\, (p+r)\left(\left(\overset{p}{\C} \oplus \overset{r}{\C} \oplus \overset{r'}{\C}\right) \underset{D}{\Asterisk}\, C\, , E^{2}_{D'}\right)(p+r)
$$
where $D' = \overset{p}{\C} \oplus \overset{r}{\C}$, the conditional expectations $E^{i}_{D'}$ onto $D'$ are the trace preserving ones, and the free product is reduced.
\end{lem}

\begin{proof}
The case $r'=0$ is tautologically true, so we assume $r'>0$.
Let 
$$
C'=\left(\overset{p}{\C} \oplus \overset{r}{\C} \oplus \overset{r'}{\C}\right) \underset{D}{\Asterisk}\, C.
$$
We first note that $A$ is generated by $B = (p+r)B$ and $C'$.  Since $p+r$ is the unit of $B$, it follows that if an alternating word in $C'$ and $B$ is compressed by $p+r$, then the compressed word is a word in $B$ and $(p+r)C'(p+r)$.

Let $B^{0}\subset B$ and $((p+r)C'(p+r))^{0}\subset (p+r)C'(p+r)$ be the respective subspaces of elements with zero expectation onto $D'$.
In addition, let $(C')^{0}\subset C'$ be the subspace of elements with zero expectation onto $D$, and notice that $((p+r)C'(p+r))^{0} \subset (C')^{0}$, and that elements in $B^{0}$ also have zero expectation onto $D$ in $A$.
Now, consider an alternating word $w = x_{1}y_{1}\cdots x_{n}y_{n}x_{n+1}$ where $x_{i} \in B^{0}$ for $2 \leq i \leq n$, $x_{1}, x_{n+1} \in D' \cup B^{0}$, and $y_{i} \in ((p+r)C'(p+r))^{0}$.
We need to show that $E_{D'}(w) = 0$.

In $C'$, let $a = (p+r) - (\alpha p + \beta q)$ where $\alpha$ and $\beta$ are chosen so that $E_{D}(a) = 0$.
Notice that $\alpha = 1$ and $\beta < 1$ since $r' > 0$.
Expectationless elements in $C'$ are therefore norm limits of sums of alternating words in $\{a\}$ and $C^{0}$ (the $D$-expectationless elements in $C$), and hence we approximate each $y_{i}$ by a linear combination of alternating words in $\{a\}$ and $C^{0}$.
Since $y_{i}$ is supported under $p+r$, we may assume that the coefficient of the single word $a$ is zero since $(p+r)a(p+r)$ has nonzero expectation onto $D'$.
Inserting these words for each $y_{i}$ in $w$, and distributing and regrouping, gives a linear combination of alternating words in $(B^{0} \cup \{a\})$ and $C^{0}$ which has expectation zero by freeness.
\end{proof}

We now use the following lemma, which is a special case of \cite[Theorems 3.5 and 4.5]{MR2782689}.

\begin{lem}[\cite{MR2782689}] \label{lem:simpleamalgamated}
Let $C_{1}$ and $C_{2}$ be unital $C^{*}$ algebras containing the unital $C^{*}$ subalgebra $D$ unitally.  Suppose each $C_{i}$ is equipped with a trace $\tr_{i}$ such that $\tr_{1}$ and $\tr_{2}$ coincide on $D$ and that there exist trace preserving conditional expectations $E^{i}_{D}$ of $C_{i}$ onto $D$.  Consider the reduced amalgamated free product with conditional expectation $E$
$$
(C,E) = (C_{1}, E^{1}_{D}) \underset{D}{\Asterisk} (C_{2}, E^{2}_{D}).
$$
Then $C$ is simple and has unique trace $\tr = \tr_{1} \circ E_{D} = \tr_{2} \circ E_{D}$ provided the following conditions hold:
\be
\item
There exist unitaries $u_{1} \in C_{1}$ and $u_{2}, u_{2}'$ in $C_{2}$ such that $E_{D}(u_{1}) = 0 = E_{D}(u_{2}) = E_{D}(u_{2}') = E(u_{2}^{*}u_{2}')$.
\item
For every $a_{1}, ..., a_{n} \in D$ with zero trace, there exists a unitary $u \in (C_{2})^{0}$ such that $E_{D}(ua_{i}u^{*}) = 0$ for each $i$.
\item
There are unitaries $w, v \in (C_{2})^{0}$ such that $E_{D}(wav) = 0$ for all $a \in D$.
\ee
Furthermore, let $G$ be the subgroup of $K_{0}(C)$ which is generated by $K_{0}(j_{1})(K_{0}(C_{1}))$ and $K_{0}(j_{2})(K_{0}(C_{2}))$ with $j_{i}: C_{i} \rightarrow C$ the canonical inclusion.  If the above three conditions hold, then
$$
K_{0}(C)^{+} \cap G = \{x \in \Gamma : \tr(x) > 0\} \cup \{0\}.
$$
\end{lem}

We now prove our desired theorem:
\begin{thm} \label{thm:SimpleTrace2}
The $C^{*}$ algebra $\cA_{0}$ is simple and has unique trace when $\Gamma$ satisfies Assumption \ref{assume:Cycle}.  Moreover, $K_{0}(\cA_{0})^{+} = \{x\in K_{0}(\cA_{0}): \tr(x) > 0\} \cup \{0\}$.
\end{thm}

\begin{proof}
Set $D = \overset{p_{\star}}{\underset{1}{\C}} \oplus \overset{p_{\alpha}}{\underset{\mu(v)}{\C}}$, and let $\e'$ be the edge in $E(\vec{\Gamma})$ affiliated to $\e$ which has $s(\e') = \alpha$ and $t(\e') = \star$.  Note that
$$
D\cS(\Gamma)D = C^{*}(p_{\star}, p_{\alpha}, g_{\e'}) \underset{D}{\Asterisk} D(\cS(\Gamma \setminus \{\e\}, \mu))D.
$$

Let $q \leq p_{\alpha}$ be the support projection of $g_{\e'}g_{\e'}^{*}$.
If the polar part of $g_{\e'}$ is not in $C^{*}(g_{\e'})$, then $\mu(\alpha) = 1$ and $q=p_{\alpha}$.
Otherwise, the polar part of $g_{\e'}$ is in $C^{*}(g_{\e'})$, so $q\in C^{*}(g_{\e'})$.
Lemma \ref{lem:compressfree2} implies
$$
( p_{\star} + q)\cS(\Gamma)( p_{\star} + q)
=
\cD \underset{D'}{\Asterisk} (p_{\star} + q)\left(\left(\overset{p_{\star}}{\underset{1}{\C}} \oplus \overset{q}{\underset{1}{\C}} \oplus \overset{p_{v} - q}{\underset{\mu(v) - 1}{\C}} \right) \underset{D}{\Asterisk} \cS(\Gamma \setminus \{\e\}, \mu) \right)(p_{\star} + q)
$$
where $D' = \overset{p_{\star}}{\C} \oplus \overset{q}{\C}$, and $\cD \cong M_{2} \otimes C[0, 1]$ if $q \neq p_{v}$ and 
$$
\cD \cong \{f : M_{2}(\C) \rightarrow [0, 1]: f \text{ is continuous and } f(0) \text{ is diagonal}\}
$$
if $q=p_{v}$. Recall that in either case, the trace on $B$ is $\tr_{M_{2}(\C)} \otimes d\lambda$ with $d\lambda$ integration with respect to Lebesgue measure.  We will show that $( p_{\star} + q)\cS(\Gamma)( p_{\star} + q)$ satisfies properties (1)-(3) of Lemma \ref{lem:simpleamalgamated}.

The projections $p_{\star}$ and $q$ are identified with the matrix units $e_{11}$ and $e_{22}$ of the copy of $M_{2}$ in $B$.
In either case, $B$ contains a Haar unitary with expectation zero, namely
$$
V = \left(\begin{array}{rr} e^{2\pi i t} & 0 \\ 0 & e^{2\pi i t}\end{array} \right).
$$
Let $C'$ be the second algebra in the free product expansion of $(p_{\star} + q)\cS(\Gamma)(p_{\star} + q)$, and let $f$ be an edge different from $\e'$ with $t(f) = \star$.
The continuous functional calculus applied to $g_{f}^{*}g_{f}$ produces a Haar unitary $u\in p_{\star}C'p_{\star}$.
Furthermore, the von Neumann algebra generated by $qC' q$ is an interpolated free group factor \cite{MR3110503}, so it follows that $qC' q$ contains a unitary $v$ of trace zero.
Therefore, $u + v$ is an expectationless unitary in $C'$, and we see that $(p_{\star} + q)\cS(\Gamma)(p_{\star} + q)$ satisfies (1) in Lemma \ref{lem:simpleamalgamated}.

To see (2), consider the unitary
$$
U = \left(\begin{array}{rr} \cos(2\pi t) & -\sin(2 \pi t) \\ \sin(2\pi t) & \cos(2 \pi t)\end{array} \right)\in B.
$$
(Note $U\in B$ regardless if $q = p_{v}$ or not.)
We observe that the traceless elements in $D'$ form a 1 dimensional subspace of $D'$ spanned by
$$
a =
\begin{pmatrix}
1 & 0 \\
0 & -1
\end{pmatrix}.
$$
Note that
$$
UaU^{*} =
\begin{pmatrix}
\cos^{2}(2\pi t) - \sin^{2}(2 \pi t) & 2\cos(2\pi t)\sin(2 \pi t) \\
2\cos(2\pi t)\sin(2 \pi t)  & \sin^{2}(2\pi t) - \cos^{2}(2 \pi t)
\end{pmatrix}
$$
which is expectationless, and thus (2) in Lemma \ref{lem:simpleamalgamated} holds.

If $b \in D'$, then $b$ is a diagonal matrix with $b_{1}$ and $b_{2}$ on the diagonal, and $VbV$ is easily seen to have expectation 0.  Thus (3) holds, so by Lemma \ref{lem:simpleamalgamated}, we have simplicity of $(p_{\star} + q)\cS(\Gamma, \mu)(p_{\star} + q)$ and hence $\cA_{0}=p_\star \cS(\Gamma, \mu)p_\star$, which is a full corner.

Finally, we have also shown that $(p_{\star} + q)\cS(\Gamma)(p_{\star} + q)$ has unique trace, and simplicity implies $[q]$ represents an element in $K_{0}(\cA_{0})$.
This implies that $(p_{\star} + q)\cS(\Gamma)(p_{\star} + q)$ is isomorphic to a compression of $M_{n}(\cA_{0})$ for some $n$.
Therefore, $M_{n}(\cA_{0})$ has unique trace by Lemma \ref{lem:fulluniquetrace} below, implying $\cA_{0}$ has unique trace.

Note that $q$  is full in $p_{\alpha}\cS(\Gamma \setminus \{\e\})p_{\alpha}$, so for each vertex $\beta$ of $\Gamma$, $[p_{\beta}]$ is represented in the $K_{0}$ group of at least one of the algebras in the free product expansion of $(p_{\star} + q)\cS(\Gamma)(p_{\star} + q)$.  Since $\cA_{0}$ is a full corner of $( p_{\star} + q)\cS(\Gamma)( p_{\star} + q)$, we have
\begin{equation*}
K_{0}(\cA_{0})^{+} = \{x \in K_{0}(\cA_{0}): \tr(x) > 0\} \cup \{0\}. \qedhere
\end{equation*}
\end{proof}

By Theorems \ref{thm:SimpleTrace1} and \ref{thm:SimpleTrace2}, $\cA_0$ is always simple with unique trace.
By Proposition \ref{prop:Simple} and Corollary \ref{cor:allME}, $\cA_k$ is always simple for all $k>0$.

We now show $\cA_k$ has unique trace for all $k>0$.
We use the following lemma, which comes from \cite{MR1473439}.
We provide a short proof for the reader's convenience.

\begin{lem}\label{lem:fulluniquetrace}
Suppose $B$ is a unital $C^*$-algebra, $p\in B$ is a full projection ($BpB=B$), and $C=pBp$.
If $C$ has unique tracial state $\tr_C$, then $B$ has at most one tracial state.
\end{lem}
\begin{proof}
Suppose $\tr_1,\tr_2$ are two tracial states on $B$.
Since $\tr_C$ is the unique tracial state on $C$, for $i=1,2$, $\tr_i|_C=\lambda_i\tr_C$ for some $\lambda_i\geq 0$.
Then if $b\in B$, we can write $b=\sum_{j=1}^n x_j p y_j$, so for $i=1,2$,
$$
\tr_i(b)
=
\sum_{j=1}^n\tr_i( x_j p y_j)
=
\sum_{j=1}^n\tr_i( py_j x_j p)
=
\lambda_i \sum_{j=1}^n\tr_C( py_j x_j p).
$$
Since $\tr_i$ is a state, $\lambda_i>0$ for $i=1,2$, so $\tr_1$ is proportional to $\tr_2$.
Since $\tr_1(1)=\tr_2(1)$, $\tr_1=\tr_2$.
\end{proof}

Note that the polar part of
$
\begin{tikzpicture}[baseline=-.1cm]
    \draw[thick] (-.4, -.4) rectangle (.4, .4);
    \draw (-.4, 0) arc (-90:0:.4cm);
    \node at (0, 0) {\scriptsize{$k$}};
\end{tikzpicture}
$
is in $A_{\infty}$ since
$
\begin{tikzpicture}[baseline=-.1cm]
    \draw[thick] (-.4, -.4) rectangle (.4, .4);
    \draw (-.2, .4) arc (-180:0:.2cm);
    \node at (0, 0) {\scriptsize{$k$}};
\end{tikzpicture}
$ 
is invertible in $\cA_{0}$ using the graded convention by  Fact \ref{fact:CupInvertible}.  
Therefore $\cA_{0} \cong p\cA_{k}p$ for some projection $p \in \cA_{k}$.  As $\cA_{k}$ is simple, $p$ is full in $\cA_k$.
Since each of the $\cA_k$ comes equipped with a tracial state, we conclude the following.

\begin{cor}
$\cA_k$ has unique trace for all $k\geq 0$.
\end{cor}

\subsection{Stable rank one}
In this section, we will show $\cA_{0}$ has stable rank 1, and use this to deduce some statements about of comparability of projections in $\cA_{\infty}$.  Of course, Theorem \ref{thm:AvitzourSR} already implies $\cA_{0}$ has stable rank 1 if $\Gamma$ is as in Theorem \ref{thm:SimpleTrace1}.  We will need to eastiblsh this fact if $\Gamma$ satisfies the Assumpton \ref{assume:Cycle}.  We begin with the theorem that will allow us to do just that.

\begin{thm} \label{thm:Amalg.sr1}
Suppose that $B_{1}$ and $B_{2}$ are unital separable $C^{*}$ algebras both containing $D = \C^{2}$ as a subalgebra.  Assume that $B_{1}$ and $B_{2}$ are equipped with faithful traces $\tr_{1}$ and $\tr_{2}$ respectively such that $\tr_{1}$ and $\tr_{2}$ agree on $D$, and there exist trace preserving conditional expectations $E^{i}_{D}$ from $B_{i}$ to $D$.  Form the reduced amalgamated free product
$$
(B, E) = (B_{1}, E^{1}_{D}) \underset{D}{\Asterisk} (B_{2}, E^{2}_{D}).
$$
Let $p$ and $q$ be the two minimal projections in $D$, and assume $\tr_{i}$ is non-normalized and has the property $\tr_{i}(p) = 1 = \tr_{i}(q)$.  Suppose $B_{1}$ contains a unitary $u_{1}$ and $B_{2}$ contains unitaries $u_{2}$ and $u_{2}'$ with $v= pvp + qvq$ for $v \in \{u_{1}, u_{2}, u_{2}'\}$ which also satisfy
$$
E(u_{1}) = 0 = E(u_{2}) = E(u_{2}') = E(u_{2}^{*}u_{2}').
$$
Then $pB p$ and $qB q$ both have stable rank 1.
\end{thm}

The proof follows from the arguments given in \cite{MR1478545}.  We will sketch how to translate these arguments to the case of amalgamation over $\C^{2}$, and we will borrow the notation from \cite{MR1478545}.

Let $\Lambda^{0}(B_{1}, B_{2})$ the set of alternating words of expectationless elements in $B_{1}$ and $B_{2}$, and note that $D + \spann(\Lambda^{0}(B_{1}, B_{2}))$ is dense in $B$.  Observe that for $i \in \{1, 2\}$, the subspaces
$$
pB_{i}p,\, pB_{i}q,\, qB_{i}p, \, \text{and } qB_{i}q
$$
are mutually orthogonal under the inner product given by $\langle x, y \rangle = \tr_{i}(y^{*}x)$, and under $E^{i}_{D}$.  For each $B_{i}$, we can form an orthonormal basis $X_{i} = \{p, q \} \cup \{x_{j}^{i}: j \geq 1\}$ so that each $x_{j}^{i}$ is in one of the subspaces above for all $j$.  Following along the lines of Lemma 2.1 of \cite{MR1478545}, we may also assume that $\spann(X_{i})$ is a dense $*$-subalgebra of $B_{i}$ for $i \in \{1, 2\}$.  Note that the requirement on the $x_{i}^{j}$ forces the elements in $X_{i}$ to be orthogonal under $E^{i}_{D}$.

Borrowing the notation of \cite{MR1478545}, $Y_{0} = \{p, q\}$ and $Y_{n}$ is the subset of alternating words in the $x_{j}^{1}$ and $x_{j}^{2}$. The fact that the $x_{j}^{i}$ are supported on either side by $p$ or $q$ and the condition $\tr(p) = \tr(q) = 1$ imply that if $w$ is a word in $Y_{n}$, then either $w = 0$ or $\| w \|_{2} = 1$.  Furthermore, distinct words in $Y_{n}$ have to be orthogonal under $E_{D}$ and also under $\tr$.

Define $Y = \cup_{n=0}^{\infty} Y_{n}$ and note that $Y$ is an orthonormal basis for $L^2(B,\tr)$ and spans a dense $*$ subalgebra of $B$.  Furthermore, note that $Y$ is the basis $X_{1} \underset{D}{*} X_{2}$.  We will try to establish a relation between $\|x\|_{2} (= \sqrt{\tr(x^{*}x)})$ and $\|x\|$ for $x \in \spann(Y_{k})$.  To start, we let $E_{k}$ be the orthogonal projection from $\spann(Y)$ to $\spann(Y_{k})$.

Lemma 3.1 from \cite{MR1478545} can be translated to our case with a slight modification.  Specifically, we have the following:
\begin{lem}
Let $v \in Y_{k}$ and $w \in Y_{l}$ and $n \geq 0$ be fixed.

\be

\item  Assume $|k-l| < n \leq k+l$.  Let $q$ be the integer satisfying $k + l - n = 2q$ or $k + l - n = 2q + 1$.  Set
\begin{align*}
&v = v_{1}xv_{2} \, \, \, \, \, v_{1} \in Y_{k-q-1} \, \, \, \, \, x \in X_{i}^{0} \, \, \, \, \, v_{2} \in Y_{q}\\
&w = w_{2}yw_{1} \, \, \, \,  w_{1} \in Y_{l - q - 1} \, \, \, \, \, y \in X_{j}^{0} \, \, \, \, \, w_{2} \in Y_{q} 
\end{align*}
Then
$$
E_{n}(vw) = \begin{cases}
\delta_{i,j} \sum_{u \in X_{i}^{0}} \tr(v_{2}w_{2})\cdot \tr(u^{*}xy)v_{1}uw_{1} &\text{ if $n$ is odd}\\
(1-\delta_{i,j}) \tr(v_{2}w_{2})v_{1}xyw_{1} &\text{ if $n$ is even.}\\
\end{cases}
$$

\item 
Assume $n = |k-l|$.  If $k - l \neq 0$, set $q = \min\{k, l\}$ meaning $k + l - n = 2q$.  Write
\begin{align*}
v &= v_{1}v_{2}, \, \, \, \, \, v_{1} \in Y_{k-q}, \, \, \, \, \, v_{2} \in Y_{q}\\
w &= w_{2}w_{1}, \, \, \, \, w_{1} \in Y_{l-q}, \, \, \, \, \, w_{2} \in Y_{q}.
\end{align*}
Note that $v_{1}$ or $w_{1}$ is in $\{p, q\}$, and $E_{n}(vw) = \tr(v_{2}w_{2})v_{1}w_{1}$.

If $k = l$, then $n = 0$, and
$$
E_{0}(vw) = \begin{cases} 0 &\text{ if } v \neq w^{*} \\
p &\text{ if } v = w^{*} \text{ and } pv = v\\
q &\text{ if } v = w^{*} \text{ and } qv = v. \end{cases}
$$

\item If $n < |k-l|$ or $n > k+l$, then $E_{n}(vw) = 0$.
\ee
\end{lem}
Given $y \in \spann(Y)$ we write $y = \sum_{i}a_{i}y_{i}$ for $a_{i} \in \C$ and $y_{i} \in Y$, and note that all but finitely many $a_{i}$ are zero.  For $i \in \{1, 2\}$, define $F_{i}(y)$ to be the set of $x \in X_{i}$ where $x$ appears in a $y_{i}$ with a nonzero coefficient.  We now state the analogue of Lemma 3.3 in \cite{MR1478545} for our case, and the proof will function exactly the same as in that paper, with the only difference being that some of the newly formed words in $ab$ could be zero.

\begin{lem} Given $y \in Y$, set
$$
K(y) = \max_{i \in \{1, 2\}}\left(\sum_{x \in F_{i}(a)} \|x\|_{\I}^{2} \right)^{1/2}.
$$
Suppose $a \in Y_{k}$ and $b \in Y_{l}$.  Then
$$
\|E_{n}(ab)\|_{2} \leq 
\begin{cases} 
\|a\|_{2}\|b\|_{2} & \text{ if } k+l-n \text{ is even } \\
K(a)\|a\|_{2}\|b\|_{2} & \text{ if } k+l-n \text{ is odd.}  
\end{cases}
$$
\end{lem}
Notice that the possibility of some of the newly formed words in $ab$ being zero does not effect this bound.  This lemma implies, via the same arguments as in Lemmas 3.4 and 3.5 of \cite{MR1478545}, the following lemma:

\begin{lem}
For each $y \in \spann(\cup_{n=0}^{k}Y_{n})$, $\|y\|_{\infty} \leq (2k+1)^{3/2}K(y)\|y\|_{2}$.
\end{lem}

Now, let $u_{1} \in B_{1}$, and $u_{2}, u_{2}'$ be the unitary elements satisfying
$$
E_{D}(u_{1}) = 0 = E_{D}(u_{2}) = E_{D}(u_{2}') = E_{D}(u_{2}^{*}u_{2}')
$$
and $v = pvp + qvq$ for $v \in \{u_{1}, u_{2}, u_{2}'\}$.  
Now assume that $\{pu_{1}p, qu_{1}q\} \subset X_{1}$ and that $\{pu_{2}p, qu_{2}q, pu_{2}'p, qu_{2}'q\} \subset X_{2}$.  
Notice that if $w \in Y$, then $vw = pvpw$ or $qvqw$ based on how the basis elements in $X_{i}$ were chosen.  
Suppose $y \in \spann(\cup_{n=0}^{k} Y_{n})$.  
Following Lemma 3.7 of \cite{MR1478545}, we set $u' = (u_{1}u_{2}^{*})^{l}$ and $w = (u_{1}u_{2}')^{l}$ where $l \geq (k+3)/2$.  It follows from their arguments that
$$
u'yw = \sum_{j=1}^{M} \alpha_{j}y_{j}
$$
where each $y_{j} \in Y$ has length at most $2l + k$, begins with either $pu_{1}p$ or $qu_{1}q$, and ends with either $pu_{2}'p$ and $qu_{2}'q$.  
Choose $m \geq (2l + k + 1)/2$, and set $r = (u_{1}u_{2})(u_{1}u_{2}')^{m}(u_{1}u_{2})$.  
Notice that since $m$ is sufficiently large, the elements $ry_{i_{1}}ry_{i_{2}}\cdots ry_{i_{n}}$ and $ry_{j_{1}}ry_{j_{2}}\cdots ry_{j_{n}}$ are either 0 or have 2-norm 1.  
If they are nonzero, they are equal if and only if $i_{k} = j_{k}$ for all $k$, and are orthogonal otherwise.  Setting $u = r'u$, we see that
$$
(uyv)^{n} = \sum_{i_{n}, ..., i_{n}}\alpha_{i_{1}}\cdots\alpha_{i_{n}}ry_{i_{1}}\cdots ry_{i_{n}},
$$
so we see that $K((uyv)^{n}) \leq K(uyv)$, and $\|(uyv)^{n}\|_{2} \leq \|uyv\|_{2}^{n}$ (since some of the words $ry_{i_{1}}\cdots ry_{i_{n}}$ may be zero even though the $y_{i_{k}}$ are not).  
This establishes the following lemma:

\begin{lem} \label{lem:2normpower}
Given $y \in \spann(Y)$, There are unitaries $u$ and $w$ in $\spann(Y)$ and a constant $K < \infty$ such that
$\|(uyw)^{n}\|_{2} \leq \|uyw\|_{2}^{n}$ and $K((uyw)^{n}) \leq K$ for all $n$.
\end{lem}

Note the slight difference between this statement and Lemma 3.7 in \cite{MR1478545}.  Denote $U(B)$ and $GL(B)$ as the unitary group and group of invertible elements of $B$ respectively.  If $r(y)$ denotes the spectral radius of $y$, then it is straightforward to see that
$$
\inf_{u \in U(B)} r(uy) \geq \dist(u, GL(B))
$$
(see the discussion before Theorem 3.8 in \cite{MR1478545}).  We are now ready to prove the theorem.

\begin{proof}[Proof of Theorem \ref{thm:Amalg.sr1}]
Following the proof of Theorem 3.8 in \cite{MR1478545}, let $y \in \spann \cup_{n=0}^{k}Y_{n}$, and $u$, and $w$ be as in the statement of Lemma \ref{lem:2normpower}.  Note that there is an $l$ with
$$
u, v \in \spann\left(\bigcup_{n=0}^{l} Y_{n}\right).
$$
This implies
$$
(uaw)^{m} \in \spann\left(\bigcup_{n=0}^{m(k + 2l)}Y_{n}\right),
$$
and hence we see
$$
\|(uyw)^{m}\| \leq K(2m(k + 2l) + 1)^{3/2}\|(uyw)^{m}\|_{2} \leq K(2m(k + 2l) + 1)^{3/2}\|uyw\|_{2}^{m}.
$$
From these inequalities, we have:
$$
\inf_{u \in U(B)} r(uy) \leq r(uyw) = \liminf_{n \rightarrow \infty} \|(uyw)^{m}\|^{1/m} \leq \|y\|_{2}
$$
which implies from the discussion above that $\dist(y, GL(B)) \leq \|y\|_{2}$ for all $y \in \spann(Y)$, and thus $\dist(pyp, GL(pB p)) \leq \|pyp\|_{2}$ for all $y \in \spann(Y)$.  
If $pB p$ did not have stable rank 1, there would be some $a \in pB p$ with $\|a\| = 1 = \dist(a, GL(pB p))$ \cite{MR951510}.  
As $a$ is a limit of elements in $p\spann(Y)p$, it follows that $\dist(a, GL(B)) \leq \|a\|_{2}$, which implies that $\|a\| = 1 = \|a\|_{2}$. 
Hence $a$ is unitary in $pB p$ which is a contradiction, as $a$ was chosen to be not invertible.  
Therefore $pB p$ has stable rank 1, and similarly, $qB q$ has stable rank 1.
\end{proof}

\begin{rem}
The authors strongly believe that this is not the most general statement that can be proven about stable rank of amalgamated free products over finite dimensional algebras.  We chose to set the traces of the minimal projections in the amalgam equal to each other in order to give an easy (and helpful) description of $B$ orthonormal bases.
\end{rem}

\begin{thm}
The algebra $\cA_{0}$ has stable rank 1.
\end{thm}

\begin{proof}
Theorem \ref{thm:AvitzourSR}, together with the proof of Theorem \ref{thm:SimpleTrace1} shows that $\cA_{0}$ has stable rank  1 whenever $\Gamma$ satisfies the hypotheses of Theorem \ref{thm:SimpleTrace1}, so we may assume that $\Gamma$ satisfies Assumption \ref{assume:Cycle}.

In this case, then consider the algebra $(p_{\star} + q)\cS(\Gamma, \mu)(p_{\star} + q)$ in the proof of Theorem \ref{thm:SimpleTrace2}.  By the proof of Theorem \ref{thm:SimpleTrace2}, the amalgamated free product expansion of $(p_{\star} + q)\cS(\Gamma, \mu)(p_{\star} + q)$ satisfies the conditions in Theorem \ref{thm:Amalg.sr1}, and $p_{\star}$ is in the amalgam.  Therefore, by Theorem \ref{thm:Amalg.sr1}, $\cA_{0}$ has stable rank 1.
\end{proof}

\begin{cor}
$\cA_{k}$ has stable rank 1 for all $k$.
\end{cor}

\begin{proof}
We note that $\cA_{0} \cong p\cA_{k}p$ for some projection $p$.  From \cite{MR2063114}, since $p$ is full, $\cA_{k}$ has stable rank 1.
\end{proof}
We will now use stable rank 1 to show that $\cA_{\I}$ has comparison of projections in the following sense:

\begin{thm}\label{thm:compareproj}
Suppose $p$ and $q$ are projections in $\cA_{\I}$ with $\Tr(p) > \Tr(q)$.  Then there exists $v \in \cA_{\I}$ such that $v^{*}v = q$ and $vv^{*} \leq p$.
\end{thm}

\begin{proof}
By Theorems \ref{thm:SimpleTrace1} and \ref{thm:SimpleTrace2}, there exists $r \in P(\cA_{\I})$ such that in $K_{0}$, $[p] - [q] = [r]$.  This means that there exist projections $q'$ and $r'$ in $\cA_{\I}$ with $q'r' = 0$, $[q'] = [q]$, $[r'] = [r]$, and $[q' + r'] = [p]$.  Since $\cA_{0}$ has stable rank 1 and $\cA_{\I} = \cA_{0} \otimes \cK(H)$, it follows from \cite{MR693043} that $q' + r'$ is Murray von Neumann equivalent to $p$, and that $q$ is Murray von Neumann equivalent to $q'$.  These observations give the desired result.
\end{proof}

\begin{rem} Note that this does \emph{not} mean that if $\tr(p) = \tr(q)$ then $p$ and $q$ are equivalent in $\cA_{\I}$.  For example, if $\cP = \TL_{\bullet}(\sqrt{2})$, then $\Tr(\jw{0}) = 1 = \Tr(\jw{2})$ but $[\jw{0}] \neq [\jw{2}]$ in $K_{0}(\cA_{0})$.
\end{rem}

Recall that the scale $\Sigma(A)$ of a $C^{*}$-algebra $A$ is defined as the following subset of $K_{0}(A)$:
$$
\Sigma(A) = \set{[p]}{p \in P(A)}.
$$
We have the following corollary, which is easily deduced from Theorems \ref{thm:SimpleTrace1}, \ref{thm:SimpleTrace2}, and \ref{thm:compareproj}.

\begin{cor} \label{cor:scale}
$\Sigma(\cA_{0}) = \set{x \in K_{0}(\cA_{0})}{ \tr(x) \in (0, 1) } \cup \{[1], [0]\}$.
\end{cor}

Principal graphs of (sub)factor planar algebras have the following property, which was pointed out to us by Noah Snyder. 

\begin{fact}
Suppose $\mu(V(\Gamma))\subset \bbQ_{>0}$, and for each $\alpha\in V(\Gamma)$, denote $\mu(\alpha)=\frac{m_\alpha}{n_\alpha}$ where $m_\alpha,n_\alpha\in \bbN$ is in lowest terms.
Then either $\mu(V(\Gamma))\subset \bbN$, or $\set{n_\alpha}{\alpha \in V(\Gamma)}$ is unbounded.
\end{fact}

We deduce the following corollary, which shows one of two extremes for the algebras $\cA_{0}$.
\begin{cor}
$\cA_{0}$ is projectionless if and only if $\mu(V(\Gamma)) \subset \bbN$.  
If $\cA_{0}$ is not projectionless, then 
$
\set{\tr(p)}{p \in P(\cA_{0})}
$
is dense in $[0, 1]$
\end{cor}

\begin{proof}
If $\mu( V(\Gamma)) \subset \bbN$, then $\set{\Tr(p)}{p \in P(\cA_{\I})} = \N$.   
Since $\Tr(p_{\star}) = 1$ and $\cA_{0} = p_{\star}\cA_{\I}p_{\star}$, the result follows.
If $\mu(V(\Gamma)) \not\subset \bbN$, there is a vertex $\alpha$ such that $\mu(\alpha)$ is either irrational, or a noninteger rational number.  
In either case, the group generated by $\mu(V(\Gamma))$, which is equal to $\tr(K_0(\cA_0))$, is dense in $\R$.  
From Corollary \ref{cor:scale}, 
$
\set{\tr(p)}{p \in P(\cA_{0})}
$
is dense in $[0, 1]$.
\end{proof}

%% file: Chapters/CStarTower.tex

\section{The $C^*$-tower}\label{sec:CStarTower}

Recall from \cite{MR2732052} that the {\rm II}$_1$-factors $(\cM_k)_{k\geq 0}$ form a Jones tower.
In this section, we analyze the tower $(\cA_k)_{k\geq 0}$ and show they form a Watatani $C^*$-tower (see Definition \ref{defn:WatataniTower}).
Of particular interest are our results on the sizes of Pimsner-Popa bases.   

\subsection{Watatani index and Pimsner-Popa bases}

We provide some background from \cite{MR996807} on the basic construction for $C^*$-algebras. 

\begin{defn}
Suppose $A\subset B$ is an inclusion of $C^*$-algebras with a conditional expectation $E:B\to A$.
A quasi-basis for $B$ over $A$ is a finite subset $\{(a_i,b_i)\}_{i=1}^n$ such that for all $x\in B$,
$$
x=\sum_{i=1}^n a_i E(xb_i) = \sum_{i=1}^n E(a_i x)b_i.
$$

The inclusion $A\subset B$ is said to be of index-finite type if there is a quasi-basis for $B$ over $A$.
In this case, the Watatani index is given by
$$
[B:A]=\sum_{i=1}^n a_ib_i,
$$
which is an invertible element of $Z(B)^+$, the positive cone of the center of $B$, and is independent of the choice of quasi-basis.
\end{defn}

\begin{defn}
Suppose $A\subset B$ is an inclusion of index-finite type.
Let $\cL_A(B)$ denote the $C^*$-algebra of adjointable operators on $B$ considered as an $A-A$ bimodule.
(Note that if $T\in \cL_A(B)$, then $T$ commutes with the right $A$-action.)
It was shown in \cite{MR996807} that the following are all equal:
\be
\item
The $C^*$-algebra $\cL_A(B)$.
\item 
The $C^*$-subalgebra $\cK_A(B)\subset \cL_A(B)$ of ``compact operators."
\item
The $C^*$-subalgebra $\langle A,e_A\rangle\subset\cL_A(B)$ generated by $A$ and the projection $e_A$ onto $A\subset B$ as an $A-A$ invariant Hilbert subbimodule. 
In this case, $ae_A b^* = \ketbra{a}{b}\in \cK_A(B)$ for all $a,b\in B$.
\ee
The above algebra is called the basic construction of $A\subset B$, denoted $\langle B,A \rangle$.
Moreover, there is a canonical isomorphism from the basic construction to:
\begin{enumerate}[(4)]
\item
The $C^*$-algebra $\pi(B)e\pi(B)$ where $\pi:B\to B(H)$ is a faithful representation on a Hilbert space $H$, and $e\in B(H)$ is a projection satisfying $ebe=E(b)e$ for all $b\in B$ such that the map $A\to B(H)$ given by $a\mapsto \pi(a)e$ is injective.
\ee
\end{defn}

\begin{rem}
In the event that the inclusion $A\subset B$ is of index-finite type, the existence of a quasi-basis is equivalent to the existence of a Pimsner-Popa basis $\{b\}\subset B$, i.e., a finite subset such that one of the following equivalent statements hold, where $e_A$ is as in the above definition:
\be
\item
$\sum_b be_A b^* = 1_{\cL_A(B)}$
\item
$x=\sum_b E(xb)b^* $ for all $x\in B$
\item
$x=\sum_b bE(b^*x) $ for all $x\in B$.
\ee
See \cite{MR561983,MR996807,MR860811} for more details.
\end{rem}

\begin{rem}
If the inclusion $A\subset B$ is of index-finite type, then $\langle B, A\rangle$ is strongly Morita equivalent to $A$ via $B$ as a $\langle B, A\rangle-A$ Hilbert bimodule.
\end{rem}

\begin{defn}\label{defn:WatataniTower}
The dual conditional expectation $E_B : \langle B, A\rangle \to B$ is given by the unique extension of the map $ae_A b\mapsto [B:A]^{-1} ab$.
If $\{b\}$ is a Pimsner-Popa basis for $B$ over $A$, then $\{b e_A [B:A]^{1/2}\}$ is a basis for $\langle B, A\rangle$ over $B$, so $B\subset \langle B, A\rangle$ is of index-finite type.

Moreover, if $[B:A]\in A$, then $[\langle B, A\rangle : B]=[B: A]\in Z(A)\cap Z(B)\cap Z(\langle B,A\rangle)$.
In this case, we define the \underline{Watatani $C^*$-tower of $A\subset B$} as the tower of $C^*$-algebras obtained from iterating the basic construction.
\end{defn}

\subsection{The Watatani $C^*$-tower of the GJS algebras}

In the tower of {\rm II}$_1$-factors $(\cM_k)_{k\geq 0}$, the diagram for the inclusion $\cM_{k}\hookrightarrow \cM_{k+1}$ is given by
$$
\begin{tikzpicture}[baseline = 0cm]
	\draw (-.8, 0)--(.8, 0);
	\draw (0, 0)--(0, .8);
	\nbox{unshaded}{(0,0)}{.4}{0}{0}{$x$}
        \node at (-.6, .2) {{\scriptsize{$k$}}};
	\node at (.6, .2) {{\scriptsize{$k$}}};
\end{tikzpicture}\, \mapsto \begin{tikzpicture}[baseline = 0cm]
	\draw (-.8, 0)--(.8, 0);
    \draw (-.8, -.55)--(.8, -.55);
	\draw (0, 0)--(0, .8);
	\nbox{unshaded}{(0,0)}{.4}{0}{0}{$x$}
        \node at (-.6, .2) {{\scriptsize{$k$}}};
	\node at (.6, .2) {{\scriptsize{$k$}}};
\end{tikzpicture}\, ,
$$
and we identify $\cM_k$ with its image in $\cM_{k+1}$ and $\cA_k$ with its image in $\cA_{k+1}$. 
The trace preserving conditional expectation $E_{k}: \cM_{k} \rightarrow \cM_{k-1}$ is given by the diagram
$$
E_{k}(y) = 
\delta^{-1}\,
\begin{tikzpicture}[baseline = 0cm]
	\draw (-.8, .2)--(.8, .2);
	\draw (0, 0)--(0, .8);
	\draw (-.4, -.2) arc(90:270: .2cm) -- (.4, -.6) arc(-90:90: .2cm);
	\nbox{unshaded}{(0,0)}{.4}{0}{0}{$y$}
	\node at (-.8, .4) {{\scriptsize{$k-1$}}};
	\node at (.8, .4) {{\scriptsize{$k-1$}}};
\end{tikzpicture}\, .
$$
The map $E_{k}$ is operator norm continuous, so $E_{k}$ takes $\cA_{k}$ onto $\cA_{k-1}$. 
The following lemma as in \cite{MR860811} is easily verified on diagrams and then on all of $\cA_{k+1}$ by continuity.

\begin{lem}[Pull-down]\label{lem:pulldown}
For all $x \in \cA_{k+1}$, $xe_{k} = \delta^{2}E_{k+1}(xe_{k})e_{k}$.
\end{lem}

\begin{prop}\label{prop:basis}
The inclusion $\cA_k\subset \cA_{k+1}$ is of index-finite type.
\end{prop}
\begin{proof}
Since $\cA_{k+1}$ is unital and simple, by \cite[Lemma 2.1.6]{MR996807}, we are guaranteed the existence of a finite subset $\{x\} \subset \cA_{k+1}$ such that $\sum_{x} xe_k x^*= 1_{k+1}$.
Lemma \ref{lem:pulldown} implies that for each of the $x\in \{x\}\subset \cA_{k+1}$, there is a $b\in \cA_k$ such that $xe_k=be_k$, which immediately gives us $\sum_{b} be_kb^{*} = 1_{k+1}$.
\end{proof}

\begin{cor}
The Watatani index $[\cA_{k}:\cA_{k-1}]$ equals the Jones index $[\cM_{k}:\cM_{k-1}]=\delta^2$.
\end{cor}
\begin{proof}
A Pimsner-Popa basis for $\cA_k$ over $\cA_{k-1}$ is also a basis for $\cM_k$ over $\cM_{k-1}$.
\end{proof}

\begin{prop}
$\cA_0'\cap \cA_n = \cP_{2n}$.
\end{prop}
\begin{proof}
$\cP_{2n} \subseteq \cA_0'\cap \cA_n =\cM_0'\cap \cA_n \subseteq \cM_0'\cap \cM_n = \cP_{2n}$ by \cite{MR2732052}.
\end{proof}

\begin{defn}
Define the projection 
$e_{k} = \delta^{-1}\,
\begin{tikzpicture} [baseline = -.1cm]
	\nbox{}{(0,0)}{.4}{0}{0}{}
	\draw (-.4, .1)--(.4, .1);
	\draw (-.4, 0) arc(90:-90:.15cm);
	\draw (.4, 0) arc(90:270:.15cm);
	\node at (0, .25) {\scriptsize{$k-1$}};
\end{tikzpicture}
\in \cA_{k+1}$.

We have $e_{k} x e_{k} = E_{k}(x)e_{k}$ for all $x \in \cA_{k}$ viewed as an element in $\cA_{k+1}$ under the identification above.
Hence $\cA_{k+1}$ is canonically isomorphic to the basic construction of $\cA_{k-1}\subset \cA_{k}$.
\end{defn}

\begin{defn}
Given a positive real number, $a$, we define $N(a) = \min\set{n \in \N}{n > a}$.
Note that if $a \in \N$, then $N(a) = a+1$.  
\end{defn}

\begin{rem}
Note that by \cite{MR860811} there exists a Pimsenr-Popa basis of $\cM_{k}$ over $\cM_{k-1}$ of sizes $N(\delta^{2})$ for $\delta^{2} \not\in \N$ and $\delta^{2}$ for $\delta^{2} \in \N$.   
Our results on comparability of projections and $K_{0}(\cA_{0})$ allow us to prove a slightly different version of this result for the $C^*$-tower.
\end{rem}

\begin{thm}\label{thm:PPSize}
Suppose $\cP_\bullet$ has modulus $\delta$.  Then for each $k$, there exists a Pimsner-Popa basis of size $N(\delta^{2})$ for $\cA_{k}$ over $\cA_{k-1}$, and the bound is sharp in the following sense.
For every $\delta$ (even when $\delta^{2} \in \N$), there exists a factor planar algebra $\cP_\bullet$ with modulus $\delta$ so that there is no Pimsner-Popa basis for $\cA_{1}$ over $\cA_{0}$ with size less than $N(\delta^{2})$.
\end{thm}

\begin{rem}
For the proof of Theorem \ref{thm:PPSize}, we will work in $\cA_\infty$.
Note that the algebras $\cA_k$ are \underline{not} identified with their images in $\cA_{k+1}$ under the inclusion above when we consider them as corners of $\cA_\infty$.
\end{rem}

\begin{proof}[Proof of Theorem \ref{thm:PPSize}]
We first note that in $\cA_{\I}$, $\tr(e_{k}) = \delta^{k-1}$ and $\tr(1_{k+1}) = \delta^{k+1} = \delta^{2}\tr(e_{k})$.  
Set $n = N(\delta^{2})$, and let $p_{1}, \dots, p_{n}$ be orthogonal projections in $\cA_{\I}$ each of which is equivalent to $e_{k}$.  
By Theorem \ref{thm:compareproj}, these exists $w \in \cA_{\I}$ with $w^{*}w = 1_{k+1}$ and $ww^{*} \leq \sum_{i=1}^{n} p_{i}$.  
Let $v_{i} \in \cA_{\I}$ be so that $v_{i}^{*}v_{i} = e_{k}$ and $v_{i}v_{i}^{*} = p_{i}$ for $1 \leq i \leq n$.  
Note that for each $i$, $w^{*}v_{i} \in \cA_{k+1}$, and
$$
\sum_{i=1}^{n} w^{*}v_{i}e_{k}v_{i}^{*}w = \sum_{i=1}^{n} w^{*}v_{i}v_{i}^{*}w = \sum_{i=1}^{n} w^{*}p_{i}w = w^{*}w = 1_{k+1}.
$$
This shows that there exist $y_{1}, \dots, y_{n} \in \cA_{k+1}$ satisfying $\sum_{i=1}^{n} y_{i}e_{k}y_{i}^{*} = 1$.
Lemma \ref{lem:pulldown} produces a Pimsner-Popa basis of size $n$.

Now, assume $\delta^{2} \not\in \N$, and let $m \in \N$ with $m < N(\delta^{2})$ (so that $m < \delta^{2}$).  
If there exists a Pimsner-Popa basis $\{b_i\}_{i=1}^m$ of size $m$, then the matrix $v\in M_{1\times m}(\cA_{k+1})$ whose $(1,i)$-th entry is $b_ie_k$ has the property that $vv^{*} = 1_{k+1}$ and $v^{*}v \leq e_{k}\otimes \id_m$, where, $e_{k}\otimes \id_m\in M_m(\cA_{k+1})$ denotes the $m\times m$ matrix with $e_{k}$ along the diagonal.  
This contradicts $\Tr(1_{k+1}) > m\Tr(e_{k})$.

If $\delta^{2} \in \N$, then the above argument also implies that there is no Pimsner-Popa basis of size $m < \delta^{2}$ for $\cA_{k}$ over $\cA_{k-1}$ when $\cP_\bullet$ has depth at least 2.  
From our computation in $K$-theory, if $\cP_\bullet$ has depth at least 2, then $[1_{2}] \neq \delta^{2}[1_{0}]$.  
If there were a Pimsner-Popa basis $b_{1}, ..., b_{\delta^{2}}$, then the above matrix trick shows that in $K_0(\cA_0)$, $[1_{2}] \geq \delta^{2}[1_{0}]$.  
Since $\Tr(1_{2}) = \delta^{2} = \delta^{2}\Tr(1_{0})$ and $\Tr$ is faithful, this would imply $[1_{2}] = \delta^{2}[1_{0}]$, a contradiction.
\end{proof}

\subsection{Examples}

In this section, we will use our results on bases, as well as the results in Section \ref{sec:properties} to state some interesting consequences about $\cA_{0}$ and the $C^{*}$-tower. 

\begin{ex}\label{ex:GroupAction}
Suppose $\cP_{\bullet}$ is the subfactor planar algebra of the group subfactor $R \subset R\rtimes G$ for $G$ a finite group of order at least 2 acting by outer automorphisms on the hyperfinite II$_{1}$ factor $R$.   
Despite the fact that $\cM_{1, +} \cong \cM_{0, +} \rtimes G$ \cite[Corollary 1.1.6]{MR1111570}, we have $\cA_{1, +} \ncong \cA_{0, +} \rtimes G$.  
Indeed, if this equality held, then there would be a Pimsner-Popa basis of unitaries $\set{u_{g}}{g \in G}$ of size $|G|$ for $\cA_{1, +}$ over $\cA_{0, +}$.  
In $K$-theory, this means $[1_{2, +}] = |G|[1_{0, +}]$, which is impossible since $\cP_{\bullet}$ has depth 2.

In particular, if $\cP_{\bullet} = \TL_{\bullet}(\sqrt{2})$, then $[\cA_{1}:\cA_{0}] = 2$, but this $K$-theoretical obstruction shows that $\cA_{1}\ncong \cA_{0}\rtimes \Z/(2\Z)$.
This is in stark contrast with Goldman's Theorem \cite{MR0107827}, which states that index two II$_1$-subfactors are always of the form $M\subset M\rtimes \bbZ/2\bbZ$.
\end{ex}

\begin{ex}\label{ex:NotME}
 If $\cP_{\bullet}$ is shaded, then for the von Neumann algebras $\cM_{k, \pm}$, we always have $\cM_{k, +} \cong \cM_{k, -}$, but  in general, $\cA_{k, +} \ncong \cA_{k, -}$.  
Indeed, if $\cP_{\bullet}$ is finite depth where $\Gamma_{+}$ and $\Gamma_{-}$ have a different numbers of vertices, e.g., $R \subset R\rtimes G$ for $G$ finite and nonabelian, then the two algebras are not even Morita equivalent as they have non-isomorphic $K_{0}$-groups.
\end{ex}

\begin{ex}
Our analysis also has an application to subalgebras of a free semicircular system.   
Let $\cP_{\bullet}$ be the planar subalgebra of $\C\langle X_{1}, \dots, X_{n} \rangle$ for $n \geq 2$ which is generated by the monomials $X_{i}X_{j}$.  
Notice that $\cP_{\bullet}$ has zero-dimensional odd box spaces.  
The principal graph of $\cP_{\bullet}$ consists of two vertices joined together by $n$ edges.

The algebra $\cA_{0}$ in this case is the $C^{*}$-subalgebra of Voiculuecu's free semicircular algebra $\tilde{\cA} = \overline{\C\langle X_{1}, \dots, X_{n}\rangle}^{\|\cdot\|}$ which is generated by the monomials $X_{i}X_{j}$.  
$\cA_{0}$ can be seen as the fixed points of the outer action of $\Z/(2\Z)$ on $\tilde{\cA}$ given by $X_{i} \mapsto -X_{i}$. Despite the fact that $\cA_{0}$ is finite and projectionless, it has a nontrivial $K_{0}$-group, namely $K_{0}(\cA_{0}) \cong \Z^{2}$.

\end{ex}